\documentclass[brochure,english,11pt]{smfbourbaki}

\usepackage[T1]{fontenc}
\usepackage[utf8]{inputenc}
\usepackage{lmodern,amssymb,amsmath,amsfonts,bm}
\usepackage[french,main=english]{babel}
\usepackage{amsmath, amssymb, amsfonts, amsthm, amsbsy}
\usepackage{enumerate}
\usepackage[dvipsnames]{xcolor}
\usepackage[unicode]{hyperref}
\hypersetup{
	urlcolor=blue,
	colorlinks=true,
	linkcolor= blue,
	citecolor=blue
}
\usepackage{comment, stmaryrd}
\usepackage[all]{xy}
\usepackage{tikz}
\usetikzlibrary{fit,shapes,calc,arrows,through,intersections,arrows.meta,decorations.pathmorphing}
\usepackage{tkz-euclide}
\theoremstyle{plain}
\newtheorem{theorem}{Theorem}[section]
\newtheorem{proposition}[theorem]{Proposition}
\newtheorem{lemma}[theorem]{Lemma}
\newtheorem{corollary}[theorem]{Corollary}

\newtheorem{fact}[theorem]{Fact}

\newtheorem{conjecture}[theorem]{Conjecture}
\theoremstyle{definition}
\newtheorem{example}[theorem]{Example}
\newtheorem{definition}[theorem]{Definition}
\newtheorem{question}[theorem]{Question}

\theoremstyle{remark}
\newtheorem{remark}[theorem]{Remark}

\newtheorem{notation}[theorem]{Notation}

\newenvironment{new-abstract}[1]
  {\bigskip\selectlanguage{#1}%
   \begin{center}\bfseries\abstractname\end{center}}
  {\par\bigskip}
\newcommand\define[1]{{\bf #1}}

\providecommand{\bfGamma}{\boldsymbol{\Gamma}}
\newcommand\fraki{\mathfrak{i}}

\newcommand\Th{\mathrm{Th}}
\newcommand\res{\mathrm{res}}
\newcommand\val{\mathrm{val}}
\newcommand\ac{\mathrm{ac}}
\newcommand\bfres{\mathbf{res}}
\newcommand\bfval{\mathbf{val}}

\newcommand\Lang{\mathfrak{L}}
\newcommand\Lring{\Lang_{\mathrm{ring}}}
\newcommand\Lval{\Lang_{\val}}
\newcommand\Lvalpi{\Lang_{\val}(\varpi)}
\newcommand\Lac{\Lang_{\ac}}
\newcommand\Loag{\Lang_{\mathrm{oag}}}

\newcommand\Llambda{\Lang_{\lambda}}

\newcommand\Lvlambda{\Lang_{\val,\lambda}}

\newcommand\Lk{\Lang_{\mathbf{k}}}
\newcommand\LGamma{\Lang_{\bfGamma}}
\newcommand\Lof{\Lang_{\mathrm{of}}}
\newcommand\Ldagger{\Lang_{p,e}}
\newcommand\squaretall\blacklozenge
\renewcommand{\square}{\mathbin{\rotatebox[origin=c]{-90}{$\squaretall$}}}
\newcommand\NN{\mathbb{N}}
\newcommand\ZZ{\mathbb{Z}}
\newcommand\QQ{\mathbb{Q}}
\newcommand\RR{\mathbb{R}}
\newcommand\CC{\mathbb{C}}
\newcommand\FF{\mathbb{F}}
\newcommand\PP{\mathbb{P}}

\newcommand\bfK{\mathbf{K}}
\newcommand\bfk{\mathbf{k}}
\newcommand\bfG{\bfGamma}
\newcommand\Gps[2]{{#1(\!(#2)\!)}}
\newcommand\iGps[2]{{#1[\![#2]\!]}}
\newcommand\Hs[2]{{\Gps{#1}{t^{#2}}}}
\newcommand\ps[1]{{\Gps{#1}{t}}}
\newcommand\ips[1]{{\iGps{#1}{t}}}
\newcommand\red{\mathrm{red}} % reduction
\newcommand\AKEtext{AKE principle}
\newcommand\AKEtextS{AKE principles}
\newcommand\AKEtransfertext{Ax--Kochen/Ershov transfer principle}
\newcommand{\AKE}{\mathrm{AKE}}
\newcommand{\sAKE}{\mathrm{s}\AKE}
\newcommand{\OAG}{\mathbf{OAG}} %%% OAGs
\newcommand{\Fth}{\mathbf{F}} %%% Fields
\newcommand{\VF}{\mathbf{VF}} %%% Valued fields
\newcommand{\ACF}{\mathbf{ACF}} %%% Algebraically closed fields
\newcommand{\RCF}{\mathbf{RCF}} %%% Real closed fields
\newcommand{\ACVF}{\mathbf{ACVF}} %%% Algebraically closed valued fields
\newcommand{\TVF}{\mathbf{TVF}} %%% Separably tame valued fields
\newcommand{\STVF}{\mathbf{STVF}} %%% Separably tame valued fields
\newcommand{\eq}{\mathrm{eq}}
\renewcommand{\H}{\mathbf{H}} %%% Henselian valued fields
\newcommand{\Hz}{\H^{0}} %%% Henselian valued fields of equal characteristic zero
\newcommand\Mod{\mathbf{Mod}}
\newcommand\perf{\mathrm{perf}}
\newcommand\ult{\mathfrak{u}}
\newcommand\ulim{\mathrm{ulim}}
\newcommand\sep{\mathrm{Sep}}
\newcommand\hens{\mathrm{h}}
\renewcommand\hens{{h}}
\newcommand\supp{\mathrm{supp}}
\newcommand\mred{\leq_{m}}
\newcommand\meq{\simeq_{m}}
\newcommand\Cperf{\mathbf{C}^{\mathrm{Perf}}} %%% PERFECTOID
\newcommand\CAT{\mathbf{C}^{\mathrm{AT}}} %%% ALMOST TAME
\newcommand\CST{\mathbf{C}^{\mathrm{ST}}} %%% SEMI-TAME
\newcommand\CAST{\mathbf{C}^{\mathrm{AST}}} %%% ALMOST SEPARABLY TAME
\newcommand\TAT{\mathbf{AT}} %%% ALMOST TAME
\newcommand\TST{\mathbf{ST}} %%% SEMI-TAME
\newcommand\ring[1]{#1^{\circ}}
\newcommand\ideal[1]{#1^{\circ\circ}}
\newcommand\Ring[1]{\mathcal{O}_{#1}}
\newcommand\Ideal[1]{\mathfrak{m}_{#1}}
\newcommand\tilt[1]{#1^{\flat}}

\newcommand\fet{\textbf{-fét}}
\usepackage[
backend=biber,
style=authoryear, 
citestyle=authoryear-comp,
maxnames=5
]{biblatex}
\usepackage{csquotes}

\DeclareDelimFormat{nameyeardelim}{\addcomma\space}

\DeclareNameAlias{sortname}{given-family}

\renewcommand{\bibnamedash}{\leavevmode\raise3pt\hbox to3em{\hrulefill}\space}

\AtEveryBibitem{%
  \clearfield{issn} % Remove issn
  \clearfield{doi} % Remove doi

  \ifentrytype{online}{}{% Remove url except for @online
    \clearfield{url}
  }
}

\renewbibmacro{in:}{%
    \ifentrytype{article}{}{\printtext{\bibstring{in}\intitlepunct}}}

\DeclareFieldFormat[article,periodical,inreference]{number}{\mkbibparens{#1}}
\DeclareFieldFormat[article,periodical,inreference]{volume}{\mkbibbold{#1}}
\renewbibmacro*{volume+number+eid}{%
    \printfield{volume}%
    \setunit*{\addthinspace}% NEW
    \printfield{number}%
    \setunit{\addcomma\space}%
    \printfield{eid}}

\DeclareFieldFormat[article,inbook,incollection]{title}{\enquote{#1}\addcomma} 

\date{Janvier 2026}
\bbkannee{78\textsuperscript{e} ann\'ee, 2025--2026}  % Commandes sp\'eciales
\bbknumero{1246}                                      % Commandes spéciales

\title[Bourbaki]{The Model Theory of Perfectoid Fields\\\textrm{[after Jahnke and Kartas]}}
\author{Sylvy Anscombe}
\address{Universit\'{e} Paris Cité, Sorbonne Universit\'{e},\\ CNRS, IMJ-PRG, F-75013 Paris, France}
\email{sylvy.anscombe@imj-prg.fr}

\begin{document}
\maketitle
%%%%%%%%%%%%
\setcounter{secnumdepth}{2}
\setcounter{tocdepth}{3}
 \tableofcontents
%%%%%%%%%%%%

\section{Introduction}

The most classical version of the Ax--Kochen/Ershov Theorem
is the following statement,
proved in 
\parencite{AxKochen-III}:
two henselian%
\footnote{This terminology from valuation theory will be defined in~\ref{section:valued_fields}.}
valued fields $(K,v)$ and $(L,w)$ of residue characteristic zero
have the same theory%
\footnote{%
	For background on first-order logic, including a careful definition of formulas,
	see~\cite{HilsLoeser} or~\cite{TentZiegler}.%
}
in a suitable first-order\footnote{We will only consider ``first order'' languages, which is to say that quantification is admitted only over elements of the structures concerned.} language of valued fields
if and only if
their residue fields
have the same theory
in the language of rings
and
their value groups
have the same theory
in the language of ordered groups.
Equivalently,
the theory
of each such valued field $(K,v)$
in the language of valued fields
is axiomatized by the background theory of henselian valued fields
of residue characteristic zero,
together with the particular sets of axioms to specify the theories of value group and residue field.
From yet another point of view, the theories of henselian valued fields of residue characteristic zero are parametrized by the first-order theories of ordered abelian groups and of fields of characteristic zero.

This theorem is at the heart of the theory 
developed in 
\parencite{AxKochen-I,AxKochen-II,AxKochen-III,Ershov65},
and vastly generalised since then.
The conclusion of the statement itself admits many variants,
and 
we refer to them collectively referred to as ``\AKEtextS''.

The original work of Ax, Kochen, and Ershov included the famous asymptotic \AKEtransfertext\ for first-order sentences
between the local fields of mixed and positive characteristics:
for every such sentence $\varphi$ there is $n\in\NN$ such that
\begin{align*}
	\QQ_{p}\models\varphi&\Leftrightarrow\ps{\FF_{p}}\models\varphi,
\end{align*}
for all $p>n$.
Here the notation ``$K\models\varphi$''
(read as ``$K$ models $\varphi$'')
simply means that $\varphi$ is true in the field $K$.
Importantly, any quantifiers in $\varphi$ are interpreted to range over elements of $K$.
See also Example~\ref{examples}(2.,4.).

The asymptotic Ax--Kochen/Ershov transfer principle in particular gives an asymptotic resolution of the following conjecture of Artin from 1936,
cf~\parencite{Artin},
for which 
a field $k$ is said to be $C_{i}$ if,
for every form $F\in k[X_{1},\ldots,X_{n}]$ of degree $d$ with $n>d^{i}$,
the equation $F(X_{1},\ldots,X_{n})=0$
has a nontrivial solution in $k^{n}$.

\begin{conjecture}[{Artin's Conjecture on forms over $\QQ_{p}$}]
	For each prime number~$p$, the field~$\QQ_{p}$ is~$C_{2}$.
\end{conjecture}
The
conjecture is false
as shown by Terjanian's counterexample \parencite{Terjanian} in degree $d=4$ and $p=2$,
generalized to arbitrary $p$ later in \cite{Terjanian2}.
See \cite{Heath-Brown} for a more detailed exposition.
The asymptotic resolution is:

\begin{theorem}[{\cite{AxKochen-I,Ershov65}}]
There is a computable%
\footnote{A function $f:\NN\rightarrow\NN$ is \define{computable} if there 
is a Turing machine that halts with output $f(n)$ on input $n$, for each $n\in\NN$, with appropriate coding of natural numbers.}
function $d\mapsto N_{d}$,
such that for any $p>N_{d}$ and any form $F\in\QQ_{p}[X_{1},\ldots,X_{n}]$ of degree $d$,
if $n>d^{2}$ the equation $F(X_{1},\ldots,X_{n})=0$ has a nontrivial solution in $\QQ_{p}^{n}$.
\end{theorem}

This can be proved directly from a few important ingredients:
Firstly it was shown by \cite{Chevalley} that each $\FF_{p}$ is $C_{1}$,
and then by \cite{Lang} that
each $\ps{\FF_{p}}$ is $C_{2}$.
Secondly,
if we denote the set of prime numbers by $\PP$
and 
if we assume the continuum hypothesis
(i.e.,~the equality $2^{\aleph_{0}}=\aleph_{1}$ of cardinal numbers),
then there is an isomorphism
\begin{align}\label{eq:isomorphism}
	\prod_{p\in\PP}\QQ_{p}/\ult&\cong\prod_{p\in\PP}\ps{\FF_{p}}/\ult,
\end{align}
for every nonprincipal ultrafilter $\ult$ on $\PP$.
(An ultrafilter $\ult$ on $\PP$ is \define{principal} if it is of the form $\{A\subseteq\PP\mid p\in A\}$ for some $p\in\PP$.)
This isomorphism is obtained by a certain structure theory of valued fields in residue characteristic zero,
building on Kaplansky's work \parencite{Kaplansky} with Ostrowski's notion of pseudo-Cauchy sequences\footnote{Also called \define{pseudo-convergent} sequences.}~\parencite{Ostrowski}.
The use of the continuum hypothesis is a red herring: the final asymptotic result is absolute, and so independent of the hypothesis.
See for example~\cite{Shelah} for further set-theoretic results on the existence of this isomorphism.
Next, it is straightforward to formalize the property $C_{2}$ as an infinite (conjunction of a) family 
of first-order sentences $\varphi_{n}$ in the language of rings,
where
$\varphi_{n}$
expresses that
there is a nontrivial solution of the equation
$F(X_{1},\ldots,X_{n})=0$
for all forms $F$ of degree $d<\sqrt{n}$ in the $n$ variables $X_{1},\ldots,X_{n}$.
The final ingredient is perhaps the most model theoretic:
according to Łoś's Theorem,
the theory of an ultraproduct\footnote{See~\ref{rem:ultraproducts} for a very brief recap on the subject of ultraproducts.} is the ``almost all'' theory of its factors.
More precisely, for any family $(k_{i})_{i\in I}$ of structures in a first-order language, and for every ultrafilter $\ult$ on $I$, we have
\begin{align*}
	\prod_{i\in I}k_{i}/\ult\models\varphi&\Leftrightarrow \{i\in I\mid k_{i}\models\varphi\}\in\ult,
\end{align*}
for every sentence $\varphi$ in the language.
It now follows that $\varphi_{n}$ is true in $\QQ_{p}$ for all but finitely many prime numbers $p$.

The fishy isomorphism in~\eqref{eq:isomorphism}
can (without set theoretic hypotheses) be replaced by an elementary equivalence
(i.e.,~an equality of theories)
\begin{align*}
	\prod_{p\in\PP}\QQ_{p}/\ult&\equiv\prod_{p\in\PP}\ps{\FF_{p}}/\ult,
\end{align*}
for nonprincipal ultrafilters $\ult$ on $\PP$:
both of these are models of the theory of henselian valued fields with, as value group, the ultrapower $\ZZ^{\ult}$ and, as residue field, the ultraproduct $\prod_{p\in\PP}\FF_{p}/\ult$, noting that the latter has characteristic zero.
Putting this together with Kaplansky's theory, the elementary equivalence follows from a standard back-and-forth argument.
This method may be refined to give a proof of the classical \AKEtext~in residue characteristic zero.

\bigskip

In an entirely different direction, the notion of the tilt---applied to perfectoid fields, rings, and spaces---gives another way to transfer information between mixed characteristic and positive characteristic.
The first striking example is the following theorem, that describes an equivalence between the category of finite separable extensions of $\QQ_{p}(p^{1/p^{\infty}})$ on the one hand, and of $\ps{\FF_{p}}^{\perf}$ on the other.

\begin{theorem}[{\cite{FontaineWintenberger}}]
	There is an isomorphism
$\mathrm{Gal}(\QQ_{p}(p^{1/p^{\infty}}))\cong\mathrm{Gal}(\ps{\FF_{p}}^{\perf})$.
\end{theorem}

This theorem builds on
the fact that appropriate quotients of the valuation rings of these valued fields are equal:
\begin{align*}
	\Ring{\QQ_{p}(p^{1/p^{\infty}})}/p
	&=
	\Ring{\ps{\FF_{p}}^{\perf}}/t.
\end{align*}
Moreover the theorem motivates the following ``transformation'' of a complete mixed characteristic valued field into one of positive characteristic:
\begin{align*}
	\widehat{\QQ_{p}(p^{1/p^{\infty}})}&\mapsto\widehat{\ps{\FF_{p}}^{\perf}}.
\end{align*}
A \define{perfectoid field} is by definition a field $K$ with a complete valuation of rank $1$,
and not discrete, of residue characteristic $p>0$, such that the Frobenius map
$x\mapsto x^{p}$ induces a surjection
$\Ring{K}/p\rightarrow\Ring{K}/p$
(if this latter condition holds, we say that $\Ring{K}$ is \define{semi-perfect}).
The \define{tilt} of a perfectoid field $K$
is the field of fractions of the inverse limit 
$\tilt{\Ring{K}}$
of copies of the ring $\Ring{K}/p$,
with transition maps
given by the Frobenius map
$x\mapsto x^{p}$:
\begin{align*}
	\tilt{R}&:=\lim_{\leftarrow}(\ldots\rightarrow R/p\rightarrow R/p\rightarrow\ldots).
\end{align*}
In his celebrated work (\cite{Scholze}), P.~Scholze geometrized the above to prove that tilting extends to an equivalence
between the category of perfectoid rings over $K$ and the category of perfectoid rings over $\tilt{K}$.

The aim of this talk is to explain a new contribution of 
F.~Jahnke and K.~Kartas,
who have introduced an elementary class $\mathcal{C}$
%of ``almost tame''
of valued fields,
containing the perfectoid fields,
and proved a range of \AKEtextS~in this context,
including for the first time valued fields admitting nontrivial defect extensions.
I present the following two ``conceptual equations'':
\begin{align*}
	&\text{AKE for henselian valued fields of residue characteristic zero}\\
	+\quad&\text{Structure theory for complete discrete valuation rings}\\
	=\quad&\text{AKE for henselian finitely ramified mixed characteristic}
\end{align*}
and
\begin{align*}
	&\text{AKE for tame valued fields}\\
	+\quad&\text{Structure theory for perfectoid fields}\\
	=\quad&\text{AKE for perfectoid fields (and more) {\em à la} Jahnke--Kartas}.
\end{align*}
From their work one obtains a nonstandard version of the Almost Purity theorem of Scholze, and the Fontaine--Wintenberger theorem.
This should
be seen in the context of
Kuhlmann's tame valued fields,
which
are another setting in which there is an Ax--Kochen/Ershov-like theory,
as well as important earlier work of Kuhlmann and Rzepka on the valuation theory of deeply ramified fields,
and of Kartas which showed the transfer of certain decidability statements via tilting.
More recently, 
\cite{RKSS}
have formalized perfectoid fields in a language of continuous logic, and they prove that tilting is then a {\em bona fide} bi-interpretation in the sense of continuous model theory.

\begin{remark}\label{rem:ultraproducts}
Let $(A_{i})_{i\in I}$ be a family 
of structures in a given language $\Lang$.
We may form the direct product
\begin{align*}
    \prod_{i\in I}A_{i}&:=\{(a_{i})_{i\in I}\mid a_{i}\in A_{i}\}
\end{align*}
in the usual way.
However, this structure generally bears little resemblance to its factors in terms of model theory.
For example, the direct product of two integral domains is never an integral domain.
For an ultrafilter $\ult$ on $I$, we introduce the following notation:
\begin{align*}
    (a_{i})\sim_{\ult}(a_{i}')\Leftrightarrow\{i\in I\mid a_{i}=a_{i}'\}\in\ult.
\end{align*}
Then $\sim_{\ult}$ is an equivalence relation on $\prod_{i\in I}A_{i}$.
If the $A_{i}$ are rings, then the quotient corresponds to the quotient by the ideal
$\mathfrak{p}=\{(a_{i})_{i\in I}\mid\{i\in I\mid a_{i}=0\}\in\ult\}$.
If the $A_{i}$ are integral domains then $\mathfrak{p}$ is prime,
and if the $A_{i}$ are fields then $\mathfrak{p}$ is maximal.
For example, an ultraproduct of fields is a field.
Łoś proved that
	\begin{align*}
		\prod_{i\in I}A_{i}/\ult\models\varphi&\Leftrightarrow\{i\in I\mid A_{i}\models\varphi\}\in\ult,
	\end{align*}
	for all $\Lang$-sentences $\varphi$.
	We may define the \define{ultralimit}
	of any $(a_{i})_{i\in I}\in\prod_{i\in I}A_{i}$,
	to be its equivalence class under $\sim_{\ult}$,
	and we denote this ultralimit by
	$\ulim_{i\rightarrow\ult}a_{i}$.
\end{remark}

\subsection{Some model theory of fields}

To find a model theoretic view of rings, fields, and valued fields we must first settle on suitable languages:
$\Lring=\{+,\cdot,-,0,1\}$
is the usual candidate.

\begin{itemize}
	\item
		The theory of $\Th(\ZZ)$ is undecidable:
		Gödel famously proved
		that there is no algorithm
		(equivalently, no Turing machine)
		to determine when an $\Lring$-sentence belongs to this theory or not,
		see for example~\cite{Gödel}.
		Similarly, also the theory
		$\Th(\QQ)$
		is undecidable, since by \cite{JuliaRobinson},
		the set
		$\ZZ$ is an $\Lring$-definable subset of $\QQ$,
		so any statement about $\ZZ$ can be rewritten into a special statement about $\QQ$.
		This is an important example of \emph{interpretability},
		which is
		the model theoretic formalism describing when one structure appears (up to isomorphism) among the definable sets (and definable quotients) in another structure.
        In this case, there is also an interpretation in the other direction,
		since $\QQ$ is interpretable by $\ZZ$, and this is easy to see:
        an isomorphic copy of $\QQ$ appears as the quotient of
        $\ZZ\times\NN_{>0}\subseteq\ZZ^{2}$
        by the definable equivalence relation
        \begin{align*}
            (a,b)\sim(c,d)&\Leftrightarrow ad-bc=0,
        \end{align*}
        together with the singleton $(0,0)$.
        Importantly, the arithmetic operations of addition, multiplication, and subtraction, as well as the two constants, are all definable on this quotient by $\Lring$-formulas interpreted in $\ZZ$.
	\item
		It was proved by
		\cite{Tarski}
		that the theories
		$\Th(\CC)$
		and
		$\Th(\overline{\FF}_{p})$
		are decidable, and axiomatized simply by their characteristic and the background theory $\ACF$ of algebraically closed fields.
		Indeed, these theories are the only completions of $\ACF$.
		Moreover, $\ACF$ admits a quantifier-elimination in $\Lring$:
		every $\Lring$-formula
		$\varphi(\bar{x})$
		(for $\bar{x}=(x_{0},\ldots,x_{n-1})$)
		is equivalent modulo $\ACF$
		to 
		another such formula
		$\varphi^{\mathrm{qf}}(\bar{x})$
		that is quantifier-free (literally: a first-order formula with no quantifiers).
		It follows (abstractly, by listing all possible proofs from the axioms) that this map
		\begin{align*}
			\varphi(\bar{x})\mapsto\varphi^{\mathrm{qf}}(\bar{x})
		\end{align*}
		is computable, relative to a suitable Gödel coding of formulas, and there are by now rather explicit and efficient algorithms.
        
		When it comes to sentences (those formulas with no free variables) we can do a little better.
		Each $\Lring$-sentence $\varphi$ is equivalent modulo $\ACF$ to a sentence $\varphi^{\chi}$,
		which is a Boolean combination of atomic sentences
		\begin{align*}
		\underbrace{1+\ldots+1}_{\text{$p$ times}}=0,
		\end{align*}
		for $p\in\PP$.
		These sentences select individual positive characteristics.
		Again, the map $\varphi\mapsto\varphi^{\chi}$ can be chosen to be computable.
	\item
		Tarski's early work in model-theoretic algebra
		(for example~\cite{Tarski})
		also included
		an axiomatization and decision procedure for
		the theory
		$\RCF=\Th(\RR)$
		of real closed fields:
		those fields in which the squares form the nonnegative cone of a total ordering
		and every odd-degree polynomial has a root.
		There is also a effective quantifier-elimination in the language of ordered fields~$\Lof=\Lang\cup\{\leq\}$,
        i.e.,~there is computable function
		\begin{align*}
			\psi(\bar{x})\mapsto\psi^{\mathrm{qf}}(\bar{x})
		\end{align*}
		mapping each $\Lof$-formula to a quantifier-free $\Lof$-formula that is equivalent modulo $\RCF$.
        We observe another occurrence of the idea of interpretability:
        an isomorphic copy of $\CC$ appears as 
        $\RR^{2}$,
        with the arithmetic operations of addition, multiplication, and subtraction again given by $\Lring$-formulas on $\RR$.
\end{itemize}

\subsection{Valued fields}
\label{section:valued_fields}

A \define{valuation} on a field $K$
	is an epimorphism
	$v:K^{\times}\rightarrow\Gamma_{v}$
	from the multiplicative group of $K$
	onto an (additive) totally ordered abelian group $\Gamma_{v}$ (an ``OAG'')
	such that~$v(x+y)\geq\min\{v(x),v(y)\}$
	for all~$x,y\in K^{\times}$.
	By convention we extend $v$ to a map~$v:K\rightarrow\Gamma_{v}\cup\{\infty\}$
	by writing $v(0)=\infty$,
	for a new symbol $\infty$.

	\begin{notation}
As one expects, a range of notation is used in the literature.
We will write
$\mathcal{O}_{v}:=\{x\in K\mid v(x)\geq0\}$ for the
\define{valuation ring}
(sometimes written~$\ring{K}$ in the literature),
	we write
	$\mathfrak{m}_{v}:=\{x\in K\mid v(x)>0\}$
	(sometimes~$\ideal{K}$)
	for the \define{valuation ideal},
	i.e.,~the unique maximal ideal of $\mathcal{O}_{v}$.
Write
	$vK:=\Gamma_{v}$
	for the \define{value group} of $v$,
and 
	$Kv=k_{v}:=\mathcal{O}_{v}/\mathfrak{m}_{v}$
	for the \define{residue field} of $v$.
	\end{notation}

\begin{example}\label{examples}
	\begin{enumerate}
\item
			$(\QQ,v_{p})$:
	the $p$-adic valuations on $\QQ$.
	For $p\in\PP$, every $x\in\QQ^{\times}$ can be written uniquely as $\tfrac{m}{n}p^{l}$ for $l,m,n\in\ZZ$ with $p,m,n$ pairwise coprime, and~$n>0$.
	We define the \define{$p$-adic valuation} $v_{p}:\QQ\rightarrow\ZZ\cup\{\infty\}$ by
			\begin{align*}
				v_{p}(x)&:=l
			\end{align*}
			and $v_{p}(0)=\infty$.
			The valuation ring is $\ZZ_{(p)}$, the maximal ideal is $p\ZZ_{(p)}$, and the residue field is $\FF_{p}$.
\item
	The \define{field of $p$-adic numbers}
			is defined to be the completion of $(\QQ,v_{p})$, which naturally admits the structure of a valued field.
			We denote it by 
			$(\QQ_{p},v_{p})$.
			Its valuation ring $\ZZ_{p}$ can be constructed as an inverse limit of the system $\ZZ/p^{n}\ZZ$, with the usual quotient maps, or explicitly as power series in the ``variable'' $p$,
			i.e.,
			\begin{align*}
				\sum_{n=0}^{\infty}a_{n}p^{n},
			\end{align*}
			with coefficients $a_{n}\in\{0,\ldots,p-1\}$ and multiplication and addition defined to take care of the ``carrying'' that we are familiar with from basic arithmetic.
			The maximal ideal is $p\ZZ_{p}$, while its value group and residue field are again $\ZZ$ and $\FF_{p}$, respectively.
\item
Given any irreducible $f\in k[t]$ we can define a valuation $v_{f}$ on $k(t)$ by analogy with~$v_{p}$.
			For $g\in k[t]$, let $v_{f}(g)$ be the maximal integer such that $f^{v_{f}(g)}$ is the highest power of $f$ dividing $g$ in $k[t]$, and $v_{f}(g/h)=v_{f}(g)-v_{f}(h)$.
The valuation ring of $v_{f}$ is the localization $k[t]_{(f)}$ of $k[t]$ at $(f)$.
\item
	The valued field of \define{formal power series} over a field $k$ is $(\ps{k},v_{t})$,
 where
			\begin{align*}
				\ps{k}&:=\Big\{	\sum_{n=N}^{\infty}a_{n}t^{n}\;\Big|\;a_{n}\in k,N\in\ZZ\Big\}
			\end{align*}
			and $v_{t}$ is defined by
			\begin{align*}
				v_{t}\Big(\sum_{n=N}^{\infty}a_{n}t^{n}\Big)&:=N,\quad\text{supposing that $a_{N}\neq0$}.
			\end{align*}
			The valuation ring is $\ips{k}=\{\sum_{n=0}^{\infty}a_{n}t^{n}\mid a_{n}\in k\}$ and the residue field is $k$.
			Indeed, $(\ps{k},v_{t})$ is the completion of $(k(t),v_{t})$.
		\item
			For any ordered abelian group $\Gamma$, 
			and any field $k$,
			the valued field of \define{Hahn series}/\define{generalized formal power series} is $(\Hs{k}{\Gamma},v_{t})$,
	where
			\begin{align*}
				\Hs{k}{\Gamma}&:=\Big\{	\sum_{\gamma\in\Gamma}a_{\gamma}t^{\gamma}\;\Big|\;a_{\gamma}\in k,\supp\Big(\sum a_{\gamma}t^{\gamma}\Big)\text{ is well-ordered}\Big\}
			\end{align*}
			and $v_{t}$ is defined by
			\begin{align*}
				v_{t}\Big(\sum_{\gamma\in\Gamma}a_{\gamma}t^{\gamma}\Big)&:=\gamma_{0},\quad\gamma_{0}=\min\Big(\supp\Big(\sum_{\gamma\in\Gamma}a_{\gamma}t^{\gamma}\Big)\Big).
			\end{align*}
	\end{enumerate}
	\end{example}

 \begin{definition}\label{def:henselian}
	 A valued field $(K,v)$ is \define{henselian} if for all monic $f\in\Ring{v}[X]$, and all $a\in\Ring{v}$, if $f(a)\in\Ideal{v}$ and $f'(a)\notin\Ideal{v}$,
	 then there exists a unique $a'\in a+\Ideal{v}$ such that $f(a')=0$.
 \end{definition}

\begin{lemma}[{Hensel's Lemma}]\label{lem:Hensel}
Complete valued fields of rank $1$ are henselian.
In particular, $(\QQ_{p},v_{p})$ and $(\ps{k},v_{t})$ are henselian, for any field $k$.
\end{lemma}

There are numerous properties equivalent to henselianity.
For example, the henselianity of a valued field $(K,v)$ is equivalent to the uniqueness of the extension of $v$ to the separable algebraic closure of $K$,
The equation $v(x)=\tfrac{1}{p}v(x^{p})$ shows that there is always a unique extension of $v$ from $K$ to its perfect hull $K^{\perf}$.

Any valued field $(K,v)$ admits a canonical \define{henselization} $(K^{\hens},v^{\hens})$,
which is a separably algebraic extension of $(K,v)$,
and is henselian,
satisfying the obvious universal property.
It coincides with the fixed field of the decomposition subgroup of the absolute Galois group $G_{K}$.

To discuss the model theory of extensions of valued fields, we introduce to more pieces of terminology.

\begin{notation}
	An extension $(L,w)/(K,v)$ of valued fields
	is
	\define{existentially closed},
	denoted
	$(K,v)\preceq_{\exists}(L,w)$,
	if
	\begin{align*}
		&(K,v)\models\exists y_{1},\ldots,y_{n}\;\varphi(a_{1},\ldots,a_{m},y_{1},\ldots,y_{n})\\
		\intertext{if and only if}
		&(L,w)\models\exists y_{1},\ldots,y_{n}\;\varphi(a_{1},\ldots,a_{m},y_{1},\ldots,y_{n}),
	\end{align*}
	for every quantifier-free $\Lval$-formula $\varphi(x_{1},\ldots,x_{m},y_{1},\ldots,y_{n})$,
	and tuple $(a_{1},\ldots,a_{m})\in K^{m}$.
	The extension is said to be
	\define{elementary},
	denoted
	$(K,v)\preceq(L,w)$,
	if
	\begin{align*}
		&(K,v)\models\varphi(a_{1},\ldots,a_{m})\\
		\intertext{if and only if}
		&(L,w)\models\varphi(a_{1},\ldots,a_{m}),
	\end{align*}
	for every $\Lval$-formula $\varphi(x_{1},\ldots,x_{m})$
	and $(a_{1},\ldots,a_{m})\in K^{m}$.
\end{notation}

We now describe the relationships between a few important examples.

\begin{example}
	\
	\begin{enumerate}
		\item
			For any field $k$,
			the henselization
			$k(t)^{\hens}$ is the relative algebraic closure of $k(t)$ in $\ps{k}$,
			and
			we have
			$(k(t)^{\hens},v_{t})\preceq_{\exists}(\ps{k},v_{t})$,
			by a theorem of Ershov~{\rm\parencite{Ershovbook}}.
			This follows from the fact that $\ps{k}/k(t)^{\hens}$ is separable, and that $k(t)^{\hens}$ is dense in $\ps{k}$ with respect to the topology induced by $v_{t}$.
		\item
			For any prime number $p$,
			the henselization
			$\QQ^{\hens}$ of $\QQ$ with respect to the $p$-adic valuation $v_{p}$ is also the relative algebraic closure of $\QQ$ in $\QQ_{p}$
			and
			$(\QQ^{\hens},v_{p}^{\hens})\preceq_{\exists}(\QQ_{p},v_{p})$.
			This can be seen by the same density argument as in~1.
			In fact
			$(\QQ^{\hens},v_{p}^{\hens})\preceq(\QQ_{p},v_{p})$,
			which is less easy but follows from the AKE theory for~$\QQ_{p}$.
	\end{enumerate}
\end{example}

\begin{lemma}[{Fundamental equality\footnote{This is unually expressed as $[L:K]\geq\sum_{w}e(w/v)f(w/v)$ and called the ``fundamental inequality''.}}]\label{lem:fundamental_inequality}
Let $(K,v)$ be a valued field of residue characteristic exponent%
\footnote{The characteristic exponent of a field $k$ is the characteristic of $k$ if this is positive, but is $1$ if the characteristic of $k$ is zero.}
$p\in\PP\cup\{1\}$
and let $L/K$ be a finite extension.
Then
\begin{align*}
 [L:K]&=\sum_{w\supseteq v}e(w/v)f(w/v)p^{d(w/v)},
\end{align*}
where $e(w/v)=(wL:vK)$
is the ramification degree
and
$f(w/v)=[Lw:Kv]$
is the inertia degree of the extension.
\end{lemma}

The quantity $d(w/v)$ is called the \define{defect}%
\footnote{Sometimes the defect is defined to be the quantity $p^{d(w/v)}$ rather than $d(w/v)$.}
	, nontrivial when it is positive (considered to be zero when $p=1$, though ill-defined by the above equation).
Accordingly, we say an extension $(L,w)/(K,v)$ is \define{defectless} when $p^{d(w/v)}=1$,
and $(K,v)$ is [\define{separably}] \define{defectless} when all of its finite [separable] extensions are defectless.
A related notion is that of an \define{immediate} extension: $(L,w)/(K,v)$ is immediate when $e(w/v)=f(w/v)=1$.
A valued field $(K,v)$ is [[\define{separably}] \define{algebraically}] \define{maximal} when it admits no nontrivial immediate [[separably] algebraic] extensions.
There are some basic implications between these properties:
\begin{align*}
	\text{Maximal}
	&\implies
	\text{henselian and defectless}\\
	&\implies
	\text{algebraically maximal}\\
	&\implies
	\text{separably algebraically maximal}\\
	&\implies
	\text{henselian}.
\end{align*}
In residue characteristic zero, since every extension is defectless, every henselian valued field is algebraically maximal.
Also, every field of Hahn series $\Hs{k}{\Gamma}$ is maximal.

The counterpart of the fundamental inequality for transcendental extensions is the following lemma.
In the lemma we denote by $\mathrm{trdeg}(L/K)$ (respectively $\mathrm{trdeg}(Lw/Kv)$)
the transcendence degree of the field extension $L/K$ (respectively $Lw/Kv$).
We also denote by $\mathrm{rrk}(wL/vK)$ the rational rank of the abelian group $wL/vK$,
i.e.,~the dimension of $\QQ\otimes(wL/vK)$ as a $\QQ$-vector space.

 \begin{lemma}[{Abhyankhar's inequality, \cite{Abhyankar}}]\label{lem:Abhyankar_inequality}
	 Let $(L,w)/(K,v)$ be an extension of valued fields.
	 Then
	 \begin{align*}
		 \mathrm{trdeg}(L/K)&\geq\mathrm{trdeg}(Lw/Kv)+\mathrm{rrk}(wL/vK),
	 \end{align*}
 \end{lemma}

\subsection{First-order languages of valued fields}

There are numerous candidates for a first-order language of valued fields.
In my view, there is no particularly canonical choice
(our syntactic choices are always rather arbitrary).
Nevertheless the literature suggests three among the leading contenders:
\begin{itemize}
\item
$\Lval^{1}=\Lring\cup\{O\}$,
i.e.,~the expansion of $\Lring$ by a unary predicate which we interpret in a valued field $(K,v)$ by the valuation ring $\Ring{v}$.
\item
$\Lval^{2}=\Lring\cup\{\leq_{v}\}$, where $\leq_{v}$ is the symbol for a binary relation interpreted by $x\leq_{v}y\Leftrightarrow v(x)\leq v(y)$.
This is one of the languages in which the theory $\ACVF$ of algebraically closed valued fields admits quantifier-elimination, see for example \parencite{Robinson_1956}.
\item
$\Lval=(\Lring^{\bfK},\Lring^{\bfk},\Loag^{\bfG};\bfval,\bfres)$,
i.e.,~the language with three sorts
$\bfK$, $\bfk$, and $\bfG$,
each equipped with
$\Lring$,
$\Lring$,
and
$\Loag$ (the language of ordered abelian groups),
Moreover, we equip $\Lval$ with two function symbols $\bfval$ (from sort $\bfK$ to $\bfG$) and $\bfres$ (from sort $\bfK$ to $\bfk$).
Of course, the first of these is interpreted as the valuation, and the second as the residue map, for example extended to be constantly zero outside of the valuation ring.
\end{itemize}

For many (but not all) purposes, these languages are equivalent:
the same sets are definable in each case.
There is a caveat:
the value group and residue field appear explicitly (as sorts) in $\Lval$,
but only as definable quotients in $\Lval^{1}$ and $\Lval^{2}$,
and moreover the complexity of the formulas defining certain definable sets may depend on the language.

In the rest of this text, I will view valued fields as $\Lval$-structures.

\subsection{Residue characteristic zero and limits as $p\rightarrow\infty$}

We begin with the $\Lval$-theory $\VF$ of valued fields.
Let $\H$ denote the $\Lval$-theory that extends $\VF$ and axiomatizes the class of henselian valued fields:
its axioms (beyond those of $\VF$) demand, for each $n\in\NN$, that the definition of ``henselianity''
(in the form of Definition~\ref{def:henselian})
holds for polynomials $f$ of degree $\leq n$.
We let $\Hz$ denote the extension of $\H$ by axioms
\begin{align*}
	v(\underbrace{1+\ldots+1}_{p\text{-times}})=0,
\end{align*}
for each $p\in\PP$.

The original theory of Ax, Kochen, and Ershov,
as introduced above,
applies to all models of the theory $\Hz$,
not just to those coming from ultraproducts of local fields.
Perhaps the most classical formulation is the following,
which we label by ``$\AKE^{\equiv}$'' simply because it relates to the binary relation of elementary equivalence, symbolically~``$\equiv$''.
Later we will see formulations that are more general insofar as they apply to wider classes of valued fields, they relate to other relations between structures, and they allow certain enrichments of the languages.

\begin{theorem}[{$\AKE^{\equiv}$ for $\Hz$}]\label{thm:AKE_basic}
Let $(K,v),(L,w)\models\Hz$.
Then
\begin{align*}
	\underbrace{(K,v)\equiv(L,w)}_{\text{in }\Lval}
	&\Longleftrightarrow
	\underbrace{vK\equiv wL}_{\text{in }\Lring}
	\text{ and }
	\underbrace{\Gamma_{v}\equiv\Gamma_{w}}_{\text{in }\Loag}.
\end{align*}
\end{theorem}
\begin{proof}[{Sketch of the proof}]
	The implication $(\Longrightarrow)$ is trivial, since the value group and residue field are interpretable in each valued field.
It remains to argue for the converse implication.
Passing to an ultrapower if necessary,
we may assume that there is a cross-section
$\chi_{v}:vK\rightarrow K^{\times}$,
i.e.,~a group homomorphism that is a right inverse of $v$.
Moreover, any model of $\Hz$ admits a subfield of representatives of its residue field,
i.e.,~a subfield $k\subseteq K$ such that the residue map $\res_{v}$ restricts to an isomorphism $k\rightarrow Kv$.
Identifying the image of $\chi_{v}$ with a multiplicative subgroup $t^{vK}\leq K^{\times}$,
we have $(k(t^{vK})^{\hens},v_{t})\subseteq(K,v)$,
which is an immediate extension, and 
the map $n\mapsto t^{n}$ is a cross-section for both $(k(t^{vK})^{\hens},v_{t})$ and $(K,v)$).
Above $K$, on the other hand, there is an immediate extension
$(K,v)\subseteq(\Hs{k}{vK},v_{t})$,
and again $n\mapsto t^{n}$ is a cross-section for both valued fields.
This latter inclusion comes from Kaplansky's theory.
Thus, it suffices to show that any immediate extension of henselian valued fields in equal characteristic zero is an elementary extension in the language $\Lval$.
For abstract model theoretic reasons (specifically, back-and-forth arguments) it suffices to show the {\em a priori} weaker statement that each such extension is existentially closed, which again follows from Kaplansky's theory
\parencite{Kaplansky}.
In residue characteristic zero every simple immediate extension of a henselian valued field is transcendental, and its isomorphism type is determined by a suitable pseudo-Cauchy sequence.
\end{proof}

Indeed, this sketch proof shows that 
every $(K,v)\models\Hz$ is existentially closed in its unique maximal immediate extension.

The argument illustrates a pattern:
we combine underlying structural principles to describe certain valued fields
(with strong hypotheses, for example~maximality or completeness)
up to isomorphism.
We then combine this with back-and-forth arguments to find isomorphisms or embeddings between sufficiently rich elementary extensions of arbitrary models of the given theory/theories.
In fact, we don't always need to work with the strong hypotheses $vK\cong wL$ and $Kv\cong Lw$.
These isomorphism arguments are often compatible with residue field and value group extensions,
meaning that there are underlying embedding lemmas, on which I will elaborate below.

Firstly, there are many corollaries, even of this basic statement:
We begin by observing that this proves the required elementary equivalence between ultraproducts of local fields.

\begin{corollary}
Let $\ult$ be a nonprincipal ultrafilter $\ult$ on $\PP$. 
We have
\begin{align*}
	\prod_{p\in\PP}(\QQ_{p},v_{p})/\ult
	&\equiv
	\prod_{p\in\PP}(\ps{\FF_{p}},v_{t})/\ult.
\end{align*}
\end{corollary}

Combined with Łoś's Theorem, we obtain the asymptotic resolution to Artin's conjecture which began this story:

\begin{corollary}
	For all $\Lval$-sentence $\varphi$ there exists $n_{\varphi}\in\NN$ such that
	\begin{align*}
		(\QQ_{p},v_{p})\models\varphi&\Leftrightarrow(\ps{\FF_{p}},v_{t})\models\varphi,
	\end{align*}
	for all $p>n_{\varphi}$.
	Moreover the function $\varphi\mapsto n_{\varphi}$ may be chosen to be computable.
\end{corollary}

Returning to the generality of $\Hz$,
we may think schematically about the completions of $\Hz$ as parametrized by the pairs of completions of $\OAG$ and of $\Fth^{0}$, the theories of ordered abelian groups and of fields of characteristic zero, respectively:
\begin{align*}
	\Hz&\overset{\sim}{\longrightarrow}\OAG\times\Fth^{0}.
\end{align*}
This even applies to those theories extending $\Hz$ that are not necessarily complete.
In other words, there is a computable function
\begin{align*}
    \varphi\mapsto\varphi^{0}
\end{align*}
from $\Lval$-sentences to the positive Boolean combinations of $\Loag$ and $\Lring$ sentences
so that
\begin{align*}
    \Hz\models\varphi\leftrightarrow\varphi^{0},
\end{align*}
i.e.,~$\varphi$ and $\varphi^{0}$ are equivalent in every model of $\Hz$.
This might be viewed as a ``sentence-by-sentence'' version of the \AKEtext~for $\Hz$.
Since this function is computable, we have the following corollary.

\begin{corollary}\label{cor:Hz_decidability}
Let $(K,v)$ be a henselian valued field of equal characteristic zero.
Then
the theory of $(K,v)$ in the language $\Lval$ of valued fields
is decidable
if and only if
the following two conditions are satisfied:
\begin{itemize}
\item
the theory of $Kv$ in the language $\Lring$ of rings is decidable;
\item
the theory of $vK$ in the language $\Loag$ of ordered abelian groups is decidable.
\end{itemize}
\end{corollary}

I wrote, above, of a general embedding principle;
here it is:

\begin{lemma}[{Underlying embedding lemma}]\label{lem:embedding_Hz}
Let $(K_{1},v_{1}),(K_{2},v_{2})\models\Hz$ be common extensions of a valued field $(K,v)$.
Suppose that $v_{1}K_{1}/vK$ is torsion-free.
Suppose moreover that we are given
an~$\Lring$-embedding
$\varphi_{\bfk}:K_{1}v_{1}\rightarrow K_{2}v_{2}$ over $Kv$,
and an~$\Loag$-embedding
$\varphi_{\bfG}:v_{1}K_{1}\rightarrow v_{2}K_{2}$ over $vK$.
Then, there exists an
elementary extension $(K_{2},v_{2})\preceq(K_{2}^{*},v_{2}^{*})$
and an embedding of
$\varphi:(K_{1},v_{1})\rightarrow(K_{2}^{*},v_{2}^{*})$
over $(K,v)$
that induces $\varphi_{\bfk}$ and $\varphi_{\bfG}$ between the residue fields and value groups.
\end{lemma}

\begin{figure}[ht]
	\begin{align*}
  	\xymatrix@=4.0em{
		(K_{1},v_{1})
        \ar@{-->}[r]^{\varphi}
&
		(K_{2}^{*},v_{2}^{*})
&
		K_{1}v_{1}
		\ar@{->}[r]^{\varphi_{\bfk}}
&
		K_{2}v_{2}
&
		v_{1}K_{1}
		\ar@{->}[r]^{\varphi_{\bfG}}
&
		v_{2}K_{2}
\\
(K,v)
	\ar@{-}[u]_{}
  	\ar@{-}[ur]^{}
&
&
		Kv
	\ar@{-}[u]_{}
	\ar@{-}[ur]_{}
&
&
		\Gamma_{v}
	\ar@{-}[u]_{}
	\ar@{-}[ur]_{}
&
  	}
	\end{align*}
  \caption{Underlying embedding lemma}
\end{figure}
This kind of embedding lemma yields a richer family of resplendent \AKEtextS.
Let
$\Lk$ be any expansion of $\Lring$,
let
$\LGamma$ be any expansion of $\Loag$.
The corresponding expansion of $\Lval$,
obtained by expanding the sort $\bfk$ to have the language $\Lk$,
expanding $\bfG$ to have $\LGamma$,
and nothing else in addition,
is denoted $\Lval(\Lk,\LGamma)$.
Such an language is a \define{$(\bfk,\bfGamma)$-expansion} of $\Lval$.

\begin{definition}\label{def:AKE}
Let $T$ be an $\Lang$-theory of valued fields,
where $\Lang=\Lval(\Lk,\LGamma)$,
and let $\square\in\{\equiv,\equiv_{\exists},\preceq,\preceq_{\exists}\}$.
We say the class $\Mod_{\Lang}(T)$ of $\Lang$-structures that are models of $T$ is an
\define{$\AKE^{\square}$-class}
for $(\Lang,\Lk,\LGamma)$ if the following holds:

Let $(K,v)$ and $(L,w)$ be two models of $T$,
and additionally suppose $(K,v)\subseteq(L,w)$ in case $\square$ is either $\preceq$ or $\preceq_{\exists}$.
Then
$(K,v)\square(L,w)$
in $\Lang$
if and only if
the following two conditions are satisfied:
\begin{itemize}
\item
$Kv\square Lw$
in $\Lk$;
\item
$vK\square wL$
in $\LGamma$.
\end{itemize}
\end{definition}

The following theorem is an expression of ``resplendent'' \AKEtextS~for $\Hz$.

\begin{theorem}[{$\AKE$ for $\Hz$}]\label{thm:Hz}
Let $\Lang=\Lval(\Lk,\LGamma)$ be a $(\bfk,\bfG)$-expansion of $\Lval$
and let $\square\in\{\equiv,\equiv_{\exists},\preceq,\preceq_{\exists}\}$.
The class $\Mod_{\Lang}(\Hz)$ of $\Lang$-structures that are models of $\Hz$
an
$\AKE^{\square}$-class
for $(\Lang,\Lk,\LGamma)$.
\end{theorem}

\noindent
Taking $\square$ to be $\equiv$,
and simplifiying to the setting
$\Lk=\Lring$
and $\LGamma=\Loag$,
we recover the classic form of the Ax--Kochen/Ershov principle.

\begin{corollary}\label{cor:Hz_resplendent_decidability}
Let $\Lang=\Lval(\Lk,\LGamma)$ be a $(\bfk,\bfG)$-expansion of $\Lval$ by countable languages $\Lk$ and $\LGamma$.
Let $(K,v)$ be an $\Lang$-structure expanding a model of $\Hz$.
Then
the $\Lang$-theory of $(K,v)$ is Turing equivalent%
\footnote{Two sets of natural numbers are \define{Turing equivalent} if each is decidable given access to an oracle-machine for the other.
The Turing sum of two sets $A,B$ may be taken to be the disjoint union of $2A$ and $2B+1$.}
to the Turing sum of the $\Lk$-theory of $Kv$ and the $\Loag$-theory of $vK$.
In particular
the $\Lang$-theory of $(K,v)$
is decidable
if and only if
the following two conditions are satisfied:
\begin{itemize}
\item
the $\Lk$-theory of $Kv$
is decidable;
\item
the $\LGamma$-theory of $vK$
is decidable.
\end{itemize}
\end{corollary}

We can strengthen the embedding lemma to removing the hypothesis that $v_{1}K_{1}$ is torsion free at the expense of 
adding to the language an angular component map,
i.e.,~a multiplicative map
$\ac:K^{\times}\rightarrow k_{v}^{\times}$
extending the restriction of $\res_{v}$ to $\Ring{v}^{\times}$.
This enables Pas's quantifier-elimination result for this theory:

\begin{theorem}[{\cite[Theorem 4.1]{Pas}}]
	The theory $\Hz$ eliminates quantifiers over $\bfK$ in $\Lac$,
    meaning every $\Lac$-formula is equivalent modulo $\Hz$ (plus suitable axioms for the angular component map) to one with quantifiers only over the sorts $\bfK$ and $\bfG$.
\end{theorem}

There is an incredible range of work developed from these underpinnings,
including
		cell decompositions (e.g.,~\cite{Pas}),
further quantifier-elimination results (e.g.,~\cite{Flenner}),
    elimination of imaginaries (e.g.,~\cite{HaskellHrushovskiMacpherson,RKV}),
		and the recent framework of Hensel minimality (see~\cite{CluckersHalupczokRideau-Kikuchi,CluckersHalupczokRideau-KikuchiVermeulen,Vermeulen}).

\section{Two other well-understood settings: finitely ramified and tame valued fields}

There are two (or perhaps three) other rather well-understood settings for the model theory of henselian valued fields, namely those that are finitely ramified (of mixed characteristic) and those that are tame (or separably tame).

\subsection{Finitely ramified henselian valued fields}
A valued field $(K,v)$ of mixed characteristic $(0,p)$ is \define{finitely ramified}
if the interval $[0,v(p))$ in the value group is finite.
The number of elements in this interval is the \define{initial ramification degree}, often denoted by~$e$.
We say $(K,v)$ is \define{unramified} if $e=1$.
The model theoretic analysis of henselian finitely ramified valued fields $(K,v)$ is (as claimed above) achieved by combining the \AKEtextS~in equal characteristic zero with a well-developed structure theory for complete discrete valuation rings of mixed characteristic.
The latter---the structure theory---is due to Cohen (e.g.~\cite{Cohen}) and numerous others including Mac Lane and Witt.
For a general reference, one may consult \cite[Chapter II]{Serre}.

The structure theory attains its cleanest form in the unramified case, with perfect residue fields:
for each prime number $p$,
there is the map
\begin{align*}
	W:\{\text{perfect fields of char $p$}\}&\rightarrow\{\text{unram.~CDVRs, mixed char.~$(0,p)$, perfect res.~field}\}\\
	k&\mapsto W[k],
\end{align*}
where $W[k]$ is the ring of Witt vectors over $k$,
which gives an equivalence of categories.
When we allow 
imperfect residue fields,
Cohen's structure theory still gives us an unramified
complete discrete valuation ring
$C[k]$ that is unique up to isomorphism (over $k$), but the isomorphism now fails to be canonical, and the categories cease to be equivalent.

In the unramified case, the above structure theory is enough,
and there is a full range of \AKEtextS.
For example, the theory of an unramified henselian valued field $(K,v)$ is axiomatized by $\H$, together with the theory of its residue field and value group (the latter necessarily discrete),
together with an axiom
\begin{align*}
 \forall x\;(v(x)\leq0\vee v(p)\leq v(x)),
\end{align*}
which expresses the condition ``unramified''.

Even in the finitely ramified (but not necessarily unramified) case,
it is possible to find \AKEtextS~reducing these theories entirely to the residue fields and value groups, at the expense of expanding those structures.
More precisely, in \cite{ADJ24}, a language $\Ldagger$ is identified which describes the structure induced on the residue field by a finitely ramified henselian valued field so that the following AKE principles hold.

\begin{theorem}[{\cite{ADJ24,}}]
Let $(K,v)$ and $(L,w)$ be two finitely ramified henselian valued fields of mixed characteristic $(0,p)$
of initial ramification $e$.
Then, we have
\begin{align*}
    \underbrace{(K,v) \equiv (L,w)}_{\text{in }\Lval} &\Longleftrightarrow
    \underbrace{Kv \equiv Lw}_{\text{in } \Ldagger} \text{ and } \underbrace{vK \equiv wL}_{\text{in }\Loag} \\[-0.4cm]
\intertext{and}
     \underbrace{(K,v) \equiv_\exists (L,w)}_{\text{in }\Lval} &\Longleftrightarrow
    \underbrace{Kv \equiv_{\exists^+} Lw}_{\text{in } \Ldagger},
\intertext{where $\equiv_{\exists}$ (respectively $\equiv_{\exists^+}$)
denotes equality of existential (respectively positive existential) theories.
In case $(K,v) \subseteq (L,w)$, we have}
     \underbrace{(K,v) \preceq_\exists (L,w)}_{\text{in }\Lval}
     &\Longleftrightarrow
    \underbrace{Kv \preceq_\exists Lw}_{\text{in }\Lring}
     \text{ and }
     \underbrace{vK \preceq_\exists wK}_{\text{in }\Loag}.
\end{align*}
\end{theorem}

Briefly removing the assumption that the ramification is finite, there is the following influential AKE-style unpublished result of van den Dries:

\begin{theorem}[{van den Dries}]
Let $(K,v)$ and $(L,w)$ be two henselian valued fields of mixed characteristic $(0,p)$.
Then we have
\begin{align*}
\underbrace{(K,v)\equiv(L,w)}_{\text{in $\Lval$}}
	\Longleftrightarrow
	&\quad\quad
\underbrace{\Ring{v}/p^{n}\equiv\Ring{w}/p^{n}}_{\text{in $\Lring$}}
\quad\forall n\in\NN,\\
	&\quad\text{and}\quad
\underbrace{(vK,v(p))\equiv(wL,w(p))}_{\text{in $\Loag(\digamma)$}},
\end{align*}
where $\Loag(\digamma)$ is the expansion of $\Loag$ by a single constant symbol $\digamma$.
\end{theorem}

There are extensions of this result by \cite{LeeLee}.
Other important results
(including
\cite{Basarab_henselian,Kuhlmann_amc,Flenner})
describe an AKE-theory of henselian valued fields of mixed characteristic with respect to their $\mathrm{RV}$-structures,
i.e.,~the quotients $K^{\times}/1+p^{n}\Ideal{v}$, for $n\in\NN$.

It is perhaps worth emphasising at this point the variety of structures that appear on the ``right hand side'' in the various \AKEtextS~under consideration.
Sometimes (as in the classic case) we are able to reduce the model theory of a valued field to its residue field and value group.
Other times, that simply does not hold, but instead we are able to reduce to one or many residue rings or $\mathrm{RV}$-structures.
In the Jahnke--Kartas principles, below, we will be again reducing to one special residue ring---namely $\Ring{v}/p$---and the value group.

\subsection{Tame valued fields}

The theory of tame valued fields was introduced and studied in
\cite{Kuhlmann_tame},
then extended to separably tame valued fields in
\cite{KuhlmannPal} and studied there and in \cite{Anscombe_lambda}.

\begin{definition}
A valued field $(K,v)$
of residue characteristic exponent $p$
is {\rm[}\define{separably}{\rm]} \define{tame}
if
\begin{itemize}
\item
it is henselian and [separably] defectless,
\item
$vK$ is $p$-divisible,
and
\item
$Kv$ is perfect.
\end{itemize}
\end{definition}
The classes of tame and separably tame valued fields are $\Lval$-axiomatizable.

The theories $\TVF$ of tame valued fields and $\STVF$ of separably tame valued fields have appeared in quite a range of recent work.
See
\cite{AF16,AJ_Henselianity,AJ_Cohen,AJ_NIP,AK,BK,Jahnke_expansions,JahnkeKartas,vanderSchaaf,Kartas_tame,Kartas_tilting,KuhlmannRzepka,Sinclair}.

\begin{theorem}[{\cite{Kuhlmann_tame}}]\label{thm:tame}
Let $\Lang=\Lval(\Lk,\LGamma)$ be a $(\bfk,\bfG)$-expansion of $\Lval$.
\begin{enumerate}
\item
Let $\square\in\{\preceq,\preceq_{\exists}\}$.
The class $\Mod(\TVF)$ of all tame valued fields 
is an $\AKE^{\square}$-class
in $\Lang$.
\item
Let $\square\in\{\equiv,\equiv_{\exists}\}$.
The class $\Mod(\TVF^{\eq})$ of all tame valued fields of equal characteristic
and an $\AKE^{\square}$-class
in $\Lang$.
\end{enumerate}
\end{theorem}

In \cite{KuhlmannPal}
the authors reformulated and strengthened
the underlying Embedding Lemma for tame valued fields to work as well for separably tame valued fields of finite (and fixed) imperfection degree, and this was then generalized more recently in \cite{Anscombe_lambda} to go through for arbitrary imperfection degrees.
For this purpose, we introduced a language
$\Llambda$ to extend $\Lring$
(and corresponding expansion $\Lvlambda$ of $\Lval$)
by including function symbols for the $p$-th roots of the coordinates of elements with respect to $p$-bases.
One then obtains
a range of ``separable \AKEtextS''.

In order to properly include the \AKEtextS~that have been proven for separably tame valued fields, we generalize the terminology of $\AKE$-classes as follows.
Let $T$ be a theory of valued fields in a language $\Lang$ that is an
$(\bfk,\bfG)$-expansion of $\Lvlambda$.
and let $\square\in\{\equiv,\equiv_{\exists},\preceq,\preceq_{\exists}\}$.
We say that the class $\Mod_{\Lang}(T)$ of models of $T$ is an
\define{$\sAKE^{\square}$-class}
if
whenever we are given two models $(K,v),(L,w)\models T$,
(where we additionally suppose that $(L,w)$ is a separable extension of $(K,v)$ in case $\square$ is either $\preceq$ or $\preceq_{\exists}$),
we have
$(K,v)\square(L,w)$
in $\Lvlambda$
if and only if
the following three conditions hold:
\begin{itemize}
\item
$Kv\square Lw$
in $\Lk$;
\item
$vK\square wL$
in $\LGamma$;
\item
$K$ and $L$ have the same elementary imperfection degree%
\footnote{The \define{elementary imperfection degree} of a field $k$ of characteristic $p$ is symbolically $\infty$ if $k$ is an infinite extension of its subfield $k^{(p)}$ of $p$-th powers, or otherwise is the natural number $\fraki$ such $p^{\fraki}=[k:k^{(p)}]$ if this is finite.
For the sake of consistency,
if $k$ is of characteristic zero, we say that the elementary imperfection degree is zero.}
\end{itemize}

\begin{theorem}[{\cite{Kuhlmann_tame,KuhlmannPal,Anscombe_lambda}}]\label{thm:STVFe}
Let $\Lang=\Lvlambda(\Lk,\LGamma)$ be a $(\bfk,\bfG)$-expansion of $\Lvlambda$
and
let $\square\in\{\equiv,\equiv_{\exists},\preceq,\preceq_{\exists}\}$.
The class $\Mod_{\Lang}(\STVF^{\eq})$ of all $\Lang$-structures which expand separably tame valued fields of equal characteristic is
an
$\sAKE^{\square}$-class.
\end{theorem}

\noindent
Again, taking $\square$ to be $\equiv$,
and simplifying to the setting
$\Lk=\Lring$
and $\LGamma=\Loag$,
we find the classic form of the Ax--Kochen/Ershov principle, now expressed for separably tame valued fields of equal characteristic.
\begin{center}{\em
	For all $(K,v),(L,w)\models\STVF^{\eq}$
	we have
	$(K,v)\equiv(L,w)$
	in $\Lvlambda$
	if and only if
	$Kv\equiv Lw$ in $\Lring$,
	$vK\equiv wL$ in $\Loag$,
	and
	$K$ and $L$ have the same elementary imperfection degree.
}\end{center}

\begin{proof}[Proof idea for case $e=\infty$]
The added difficultly in this proof over that of Theorem~\ref{thm:tame} is that when faced with an embedding problem,
we must ensure that $K_{2}$ is separable extension of the image of $\varphi$.
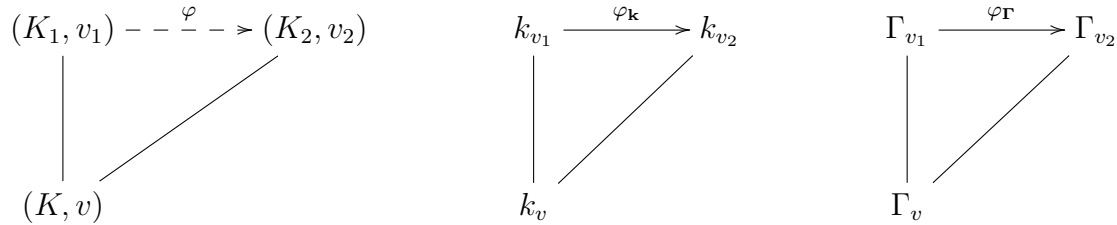
\begin{align*}
  	\xymatrix@=4.0em{
		(K_{1},v_{1})
        \ar@{-->}[r]^{\varphi}
&
		(K_{2},v_{2})
&
		k_{v_{1}}
		\ar@{->}[r]^{\varphi_{\bfk}}
&
		k_{v_{2}}
&
		\Gamma_{v_{1}}
		\ar@{->}[r]^{\varphi_{\bfG}}
&
		\Gamma_{v_{2}}
\\
		(K,v)
	\ar@{-}[u]_{}
  	\ar@{-}[ur]^{}
&
&
		k_{v}
	\ar@{-}[u]_{}
	\ar@{-}[ur]_{}
&
&
		\Gamma_{v}
	\ar@{-}[u]_{}
	\ar@{-}[ur]_{}
&
  	}
	\end{align*}
    After various standard reductions, one is left to prove that a single element $a\in K_{1}$---which may be assumed to be the pseudo-limit of a pseudo-Cauchy sequence of transcendental type over $K$---may be mapped into $K_{2}$ in such a way as to preserve the function symbols in the language $\Llambda$.
    For this one applies Kaplansky's theory to know that we have a nontrivial valuative ball of possible choices, and only a proper linear subspace to avoid.
\end{proof}

As above, we have the example corollary:

\begin{corollary}\label{cor:STVFe_decidability}
Let $(K,v)$ be a separably tame valued field of equal characteristic.
Then
\begin{itemize}
\item
the theory of $(K,v)$ in the language $\Lval$ of valued fields
is decidable
\end{itemize}
if and only if
\begin{itemize}
\item
the theory of $Kv$ in the language $\Lring$ of rings is decidable and
\item
the theory of $vK$ in the language $\Loag$ of ordered abelian groups is decidable.
\end{itemize}
\end{corollary}

In the following example, we make use of the notion of ``many-one equivalence''. between two theories (each conflated with its image in $\NN$ under a suitable Gödel coding).
Such a many-one equivalence
(denoted $A\meq B$ for sets $A,B\subseteq\NN$)
means that $A$ is many-one reducible to $B$
(written $A\mred B$),
which in turn means that there is a computable function $f:\NN\rightarrow\NN$ such that $A=f^{-1}(B)$,
and that $B$ is many-one reducible to $A$.

\begin{example}
For $p$-divisible $\Gamma$, and perfect $k$, we have
\begin{itemize}
\item
$\Th(\Hs{\FF_{q}}{\Gamma},v_{t})\meq\Th(\Gamma)$,
\item
$\Th(\Hs{\overline{\FF}_{p}}{\Gamma},v_{t})\meq\Th(\Gamma)$,
\item
$\Th(\Hs{\overline{k}_{p}}{\QQ},v_{t})\meq\Th(k)$,
\item
$\Th(\Hs{\overline{k}_{p}}{p^{-\infty}\ZZ},v_{t})\meq\Th(k)$, and
\item
$\Th(\Hs{\FF_{q}}{\QQ},v_{t})$ is decidable.
\end{itemize}
\end{example}

\subsection{Singletons, pseudo-uniformizers, and the standard decomposition}

\begin{definition}[{Pointed valued fields}]\label{def:pointed_valued_fields}
	A \define{pointed valued field} is a triple $(K,v,\pi)$ where $(K,v)$ is a valued field and $\pi\in\Ideal{v}\setminus\{0\}$.
    Such an element $\pi$ may also be called a pseudo-uniformizer.
\end{definition}

\begin{remark}\label{rem:standard_decomposition}
Given a pointed valued field $(K,v,\pi)$,
the \define{standard decomposition} 
	is the following decomposition (in the sense of composition and decomposition of places):
let $\Delta_{\pi-}$ be the largest convex subgroup of $vK$ that does not contain $v(\pi)$,
and let $\Delta_{\pi+}$ be the smallest convex subgroup of $vK$ that does contain $v(\pi)$.
Then the quotient $vK\rightarrow vK/\Delta_{\pi-}$ corresponds to a coarsening $v_{\pi-}$ of $v$,
and the quotient $vK\rightarrow vK/\Delta_{\pi+}$ corresponds to a coarsening $v_{\pi+}$ of $v_{\pi-}$.
We also observe that $v_{\pi-}$ induces a rank $1$ valuation $\bar{v}_{\pi-}$ on the residue field $Kv_{\pi+}$.
\end{remark}
\begin{figure}[ht]
  \begin{align*}
  	\xymatrix@=4.0em{
		K^{*}
  	\ar@{->}[r]^{}
&
		K^{*}(v^{*})_{\pi+}
  	\ar@{->}[r]^{}
&
		K^{*}(v^{*})_{\pi-}
  	\ar@{->}[r]^{}
&
		K^{*}v^{*}
\\
&
		K
	\ar@{-}[u]_{}
	\ar@{-}[ul]_{}
  	\ar@{->}[r]^{}
&
		Kv
	\ar@{-}[u]_{}
	\ar@{-}[ur]_{}
&
  	}
  \end{align*}
  \caption{Standard decomposition at $\pi$}
\end{figure}
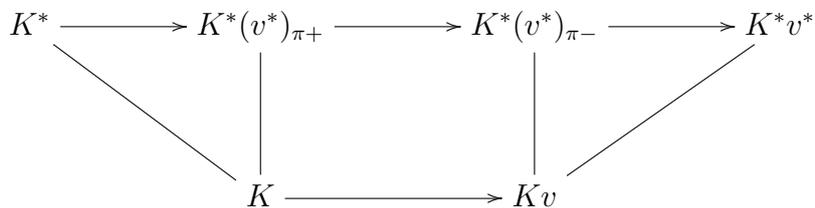

\begin{fact}
When $(K^{*},v^{*},\pi)$ is $\aleph_{1}$-saturated,
$(K^{*}(v^{*})_{\pi+},(\bar{v}^{*})_{\pi-})$
is maximal, complete, and valued in either $\ZZ$ or $\RR$.
\end{fact}

\begin{example}
$\Th(\Hs{\FF_{q}}{\QQ},v_{t},t)$ is decidable: this is a result from \cite{Lisinski} that uses Kedlaya's computable enumeration of the relative algebraic closure of $\FF_{p}(t)$ in $\Hs{\FF_{p}}{\QQ}$.
\end{example}

\section{Taming perfectoid fields: the class of ``almost tame'' valued fields}

We fix a prime number $p\in\PP$.
A ring $R$ of characteristic $p$ is \define{semiperfect} if
the Frobenius map $R\rightarrow R$, given by $x\mapsto x^{p}$, is surjective.
            
\begin{definition}\label{def:perfectoid}
	A valued field $(K,v)$ of residue characteristic $p>0$ is \define{perfectoid} if
	\begin{enumerate}
		\item
			it is of rank $1$ and complete,
		\item
			the value group $vK$ is $p$-divisible, and
		\item
			the ring $\mathcal{O}_{v}/p$ is semiperfect.
	\end{enumerate}
\end{definition}
We denote by
$\Cperf$
the class of perfectoid fields.

\begin{definition}\label{def:tilt}
	The \define{tilt} $(K,v)^{\flat}$ of a perfectoid field $(K,v)$ is
    the field of fractions of the valuation ring
	$$\Ring{\tilt{v}}:=\varprojlim_{x\mapsto x^{p}}\mathcal{O}_{v}/p\mathcal{O}_{v}.$$
\end{definition}

The primary source material is certainly
\cite{Scholze},
but there are several other important surveys and references, 
including
\cite{Fontaine_Bourbaki,Morrow_Bourbaki,Morrow_raconte}.
% FarguesFontaine

Tilting is certainly not injective,
as a map from mixed characteristic to positive characteristic perfectoid fields.
Untilts of a given perfectoid field
of positive characteristic are parametrized
by distinguished elements $\xi\in W[\Ring{\tilt{v}}]$,
up to a unit.
See the exposition in \cite{Kartas_tilting}.

\begin{definition}[{``Almost tame fields'', \cite[Definition 4.2.1]{JahnkeKartas}}]\label{def:almosttame}
Let $\CAT_{p}$ be the class of pointed valued fields $(K,v,\pi)$
of residue characteristic $p$
such that:
	\begin{enumerate}[{\bf(i)}]
		\item
			The valued field $(K,v)$ is henselian.
		\item
			The ring $\Ring{v}/p$ is semiperfect.
		\item
			The valuation ring $\Ring{v}[\pi^{-1}]$ is algebraically maximal.
\end{enumerate}
\end{definition}

Elements of $\CAT_{p}$ will naturally be construed as $\Lvalpi$-structures.

\begin{proposition}[{\cite[Proposition 4.2.2]{JahnkeKartas}}]\label{prp:elementary}
The class $\CAT_{p}$ is $\Lval$-elementary.
\end{proposition}
\begin{proof}[Sketch of the proof]
Axioms {\bf(i,ii)} are obviously elementary in the language $\Lval$.
It remains to argue that {\bf(iii)} is likewise elementary, modulo the background theory of henselian pointed valued fields of residue characteristic $p$, with $\Ring{v}/p$ semiperfect.
When $v$ is henselian and $\Ring{v}/p$ is semiperfect, the following are equivalent:
\begin{itemize}
\item
$v_{\pi+}$ is algebraically maximal
\item
every unramified finite extension of $(K,v)$ is generated by a singleton $a$
satisfiying the inequality~$0\leq v(\delta(a))\leq v(\pi)$,
where $\delta(a)$ is the different of $a$ over $K$.
\end{itemize}
The latter is equivalent to the conjunction of an infinite family of $\Lval$-sentences, thus is axiomatizable.
The equivalence itself may be proved directly.
\end{proof}

\begin{remark}[{\cite[Remark 4.2.4]{JahnkeKartas}}]
\
\begin{enumerate}[{\bf(a)}]
\item
	If $(K,v)$ is perfectoid and $\pi\in\Ideal{v}\setminus\{0\}$ then $(K,v,\pi)\in\CAT_{p}$.
	So ``perfectoid implies almost tame''.
\item
	Let $(K,v)$ be mixed characteristic $(0,p)$ and not unramified.
	Then $(K,v,p)\in\CAT_{p}$ if and only if $(K,v)$ henselian and roughly deeply ramified.
	Deeply ramified valued fields have no discrete archimedean components.
\item
	The semiperfectness of $\Ring{v}/p$ implies that $K$ is perfect.
	Deeply ramified henselian fields need not be perfect, but their completions necessarily are.
\item
	Not every henselian deeply ramified field of positive characteristic admits~$\pi\in\Ideal{v}\setminus\{0\}$ such that $\Ring{v}[\pi^{-1}]$ is algebraically maximal.
\end{enumerate}
\end{remark}

\begin{lemma}[{\cite[Lemma 4.2.5]{JahnkeKartas}}]
	Let $(K,v,\pi)\in\CAT_{p}$.
	Let $w$ be any coarsening of $v$ with $w(\pi)=0$.
	Then $(Kw,\bar{v},\res_{w}(\pi))\in\CAT_{p}$,
    where $\bar{v}$ is the valuation induced on $Kw$ by $v$.
\end{lemma}

\section{Ax--Kochen/Ershov principles for almost tame valued fields}

First we need a definition, to prepare for a special case of the main theorems.

\begin{definition}[{\cite{RobinsonZakon}}]
An ordered abelian group $\Gamma$ is \define{regularly dense}
if for all
$\gamma_{1},\gamma_{2}\in\Gamma$ with
$\gamma_{1}<\gamma_{2}$,
and all $n\in\NN$,
there is $\gamma\in\Gamma$
with $\gamma_{1}<n\gamma<\gamma_{2}$.
\end{definition}

\cite{RobinsonZakon} proved that an ordered abelian group $\Gamma$ is regularly dense if and only if it is elementarily equivalent to a dense Archimedean ordered abelian group.

As preparation, \cite{JahnkeKartas} take the step of developing a range of Ax--Kochen/Ershov results for the class of ``roughly tame'' valued fields, which is the closure of the class of tame valued fields under composition.
Rather than dwell on that, we will get straight to the first AKE principle for almost tame valued fields:

\begin{theorem}[{AKE relative subcompleteness, \cite[Theorem 5.1.2]{JahnkeKartas}}]\label{thm:AKE-relative_subcompleteness}
	Let $K_{0}\subseteq K_{1},K_{2}$ be pointed henselian valued fields of mixed characteristic $(0,p)$
	such that
	\begin{enumerate}[{\bf(a)}]
\item
	the rings $\mathcal{O}_{0}/p$,
	$\mathcal{O}_{1}/p$,
	and
	$\mathcal{O}_{2}/p$
	are semiperfect,
\item
	the valuation rings $\mathcal{O}_{0}[\pi^{-1}]$,
	$\mathcal{O}_{1}[\pi^{-1}]$,
	and
	$\mathcal{O}_{2}/[\pi^{-1}]$
	are algebraically maximal,
	and
\item
	one has $\Gamma_{0}\preceq_{\exists}\Gamma_{1}$ in $\Loag$.
	\end{enumerate}
	Then 
\begin{align*}
	\underbrace{K_{1}\equiv_{K_{0}}K_{2}}_{\text{in }\Lval}
	&\Longleftrightarrow
	\underbrace{\mathcal{O}_{1}/\pi\equiv_{\mathcal{O}_{0}/\pi}\mathcal{O}_{2}/\pi}_{\text{in }\Lring}
	\text{ and }
	\underbrace{\Gamma_{1}\equiv_{\Gamma_{0}}\Gamma_{2}}_{\text{in }\Loag}.
\end{align*}
Moreover, if $\Gamma_{0}$ is regularly dense,
	then the condition {\bf(c)} can be omitted.
	If both $\Gamma_{1}$ and $\Gamma_{2}$ are regularly dense, then 
\begin{align*}
	\underbrace{K_{1}\equiv_{K_{0}}K_{2}}_{\text{in }\Lval}
	&\Longleftrightarrow
	\underbrace{\mathcal{O}_{1}/\pi\equiv_{\mathcal{O}_{0}/\pi}\mathcal{O}_{2}/\pi}_{\text{in }\Lring}.
\end{align*}
\end{theorem}
\begin{proof}[Sketch of the proof]
	If $v(p)$ is minimal positive in $\Gamma_{v}$ then
	the result follows from known Ax--Kochen/Ershov principles for mixed characteristic unramified henselian valued fields.
Otherwise, we may suppose instead that $v(p)$ is not minimal positive.
In this case the $K_{i}$ are deeply ramified.

	$(\Rightarrow)$
Both $\mathcal{O}_{v}/\pi$ and $\Gamma_{v}$ are interpretable in $(K,v,\pi)$.

	$(\Leftarrow)$
By the Keisler--Shelah Theorem,
	we may replace $K_{0}$, $K_{1}$, and $K_{2}$
	by nonprincipal ultrapowers $K_{0}^{\ult}$, $K_{1}^{\ult}$, $K_{2}^{\ult}$
	with respect to the same ultrafilter $\ult$,
	if necessary,
	and thus assume that
	$K_{0}$ is $\aleph_{1}$-saturated
	and that there are isomorphisms
	$\varphi:\mathcal{O}_{1}/\pi\rightarrow\mathcal{O}_{2}/\pi$ over $\mathcal{O}_{0}/\pi$
	and
	$\psi:\Gamma_{1}\rightarrow\Gamma_{2}$ over $\Gamma_{0}$.
	For $i=0,1,2$, let $w_{i}$ be the coarsest coarsening of $v_{i}$ such that $w_{i}(\pi)>0$.
Then
$\mathcal{O}_{\bar{v}_{i}}=(\mathcal{O}_{v_{i}}/\pi)_{\red}$.
The isomorphism $\varphi$ induces an isomorphism
$\bar{\varphi}:\mathcal{O}_{\bar{v}_{1}}\rightarrow\mathcal{O}_{\bar{v}_{2}}$ over $\mathcal{O}_{\bar{v}_{0}}$.
This then lifts to an $\Lval$-isomorphism
$$\varphi':(k_{w_{1}},\bar{v}_{1})\rightarrow(k_{w_{2}},\bar{v}_{2})$$
over $k_{w_{0}}$.
Likewise the isomorphism $\psi$ induces an $\Loag$-isomorphism
	$$\psi':\Gamma_{w_{1}}\rightarrow\Gamma_{w_{2}}$$
	over $\Gamma_{w_{0}}$.
	We observe that $k_{w_{1}}/k_{w_{0}}$ is separable and $\Gamma_{w_{1}}/\Gamma_{w_{0}}$ is torsion-free.
	The following diagram represents the situation.
	\begin{align*}
  	\xymatrix@=4.0em{
		(K_{1},w_{1})
  	\ar@{-->}[r]^{}
&
		(K_{2},w_{2})
&
		k_{w_{1}}
		\ar@{->}[r]^{\varphi'}
&
		k_{w_{2}}
&
		\Gamma_{w_{1}}
		\ar@{->}[r]^{\psi'}
&
		\Gamma_{w_{2}}
\\
		(K_{0},w_{0})
	\ar@{-}[u]_{}
  	\ar@{-}[ur]^{}
&
&
		k_{w_{0}}
	\ar@{-}[u]_{}
	\ar@{-}[ur]_{}
&
&
		\Gamma_{w_{0}}
	\ar@{-}[u]_{}
	\ar@{-}[ur]_{}
&
  	}
	\end{align*}
    This is now an embedding problem of roughly tame valued fields.
\end{proof}

This yields---fairly quickly---three AKE principles, for $\square\in\{\preceq,\equiv,\preceq_{\exists}\}$.

\begin{theorem}[{$\AKE^{\preceq}$, \cite[Theorem 5.1.4]{JahnkeKartas}}]\label{thm:AKE-elementary_substructure}
Let $K_{1}\subseteq K_{2}$ be two pointed henselian valued fields of residue characteristic $p$.
Suppose that
	{\bf(a)}
both
$\mathcal{O}_{1}/p$ and $\mathcal{O}_{2}/p$ are semiperfect,
and that
	{\bf(b)}
both
$\mathcal{O}_{1}[\pi^{-1}]$ and $\mathcal{O}_{2}[\pi^{-1}]$
are algebraically maximal.
Then
\begin{align*}
	\underbrace{K_{1}\preceq K_{2}}_{\text{in }\Lval}
	&\Longleftrightarrow
	\underbrace{\mathcal{O}_{1}/\pi\preceq\mathcal{O}_{2}/\pi}_{\text{in }\Lring}
	\text{ and }
	\underbrace{\Gamma_{1}\preceq\Gamma_{2}}_{\text{in }\Loag}.
\end{align*}
	Moreover, if both $\Gamma_{1}$ and $\Gamma_{2}$ are regularly dense, then
\begin{align*}
	\underbrace{K_{1}\preceq K_{2}}_{\text{in }\Lval}
	&\Longleftrightarrow
	\underbrace{\mathcal{O}_{1}/\pi\preceq\mathcal{O}_{2}/\pi}_{\text{in }\Lring}.
\end{align*}
\end{theorem}
\begin{proof}
	Take $K_{0}=K_{1}$ in Theorem~\ref{thm:AKE-relative_subcompleteness}.
\end{proof}

The following two theorems also follow cleanly from Theorem~\ref{thm:AKE-relative_subcompleteness}.

\begin{theorem}[{$\AKE^{\equiv}$ in positive characteristic, \cite[Theorem 5.1.7]{JahnkeKartas}}]\label{thm:AKE-elementary_equivalence}
	Let
	$K_{1}$ and $K_{2}$ be two perfect pointed henselian valued fields 
	extending
	$(\FF_{p}(\pi),v_{\pi},\pi)$.
Suppose that
	both $\mathcal{O}_{1}[\pi^{-1}]$ and $\mathcal{O}_{2}[\pi^{-1}]$ are algebraically maximal.
Then
\begin{align*}
	\underbrace{K_{1}\equiv K_{2}}_{\text{in }\Lval}
	&\Longleftarrow
	\underbrace{\mathcal{O}_{1}/\pi\equiv\mathcal{O}_{2}/\pi}_{\text{in }\Lring}
	\text{ and }
	\underbrace{(\Gamma_{1},v(\pi))\equiv(\Gamma_{2},v(\pi))}_{\text{in }\Loag(\digamma)}.
\end{align*}
\end{theorem}

\begin{theorem}[{$\AKE^{\preceq_{\exists}}$, \cite[Theorem 5.1.10]{JahnkeKartas}}]\label{thm:AKE-existential_closedness}
	Let $K_{1}$ be a pointed henselian valued field of residue characteristic $p$.
	Suppose that {\bf(a)}
    the ring
	$\mathcal{O}_{1}/p$ is semiperfect
	and that {\bf(b)}
    the valuation ring
	$\mathcal{O}_{1}[\pi^{-1}]$ is algebraically maximal.
	Let $K_{2}$ be any pointed valued field extending $K_{1}$.
Then
\begin{align*}
	\underbrace{K_{1}\preceq_{\exists}K_{2}}_{\text{in }\Lval}
	&\Longleftrightarrow
	\underbrace{\mathcal{O}_{1}/\pi\preceq_{\exists}\mathcal{O}_{2}/\pi}_{\text{in }\Lring}
	\text{ and }
	\underbrace{(\Gamma_{1},v(\pi))\preceq_{\exists}(\Gamma_{2},v(\pi))}_{\text{in }\Loag(\digamma)}.
\end{align*}
Moreover if $\Gamma_{1}$ is regularly dense
then 
\begin{align*}
	\underbrace{K_{1}\preceq_{\exists}K_{2}}_{\text{in }\Lval}
	&\Longleftrightarrow
	\underbrace{\mathcal{O}_{1}/\pi\preceq_{\exists}\mathcal{O}_{2}/\pi}_{\text{in }\Lring}.
\end{align*}
\end{theorem}

\begin{corollary}[{\cite[Corollary 5.2.2]{JahnkeKartas}}]\label{cor:AKE}
Let $k$ be a perfect field of characteristic $p>0$.
We have the following new examples:
	\begin{enumerate}[{\bf(i)}]
		\item
			$(k(t^{p^{-\infty}})^{h},v_{t})\preceq(\ps{k}^{\perf},v_{t})\preceq(\widehat{\ps{k}^{\perf}},v_{t})$,
		\item
			$(\ps{k}^{\perf},v_{t})\preceq(\Gps{\ps{k}^{\perf}}{z^{\QQ}},v_{z}\circ v_{t})$,
            and
 		\item
			$(\ps{\varinjlim k_{i}}^{\perf},v_{t})\preceq(\ps{\bar{k}}^{\perf},v_{t})$,
			where $k_{i}$ runs over the directed system of finite extensions of $k$ and $\bar{k}$ is the algebraic closure of $k$.
	\end{enumerate}
    As a particular case of {\bf(i)} we have
    \begin{align*}
        \FF_{p}(t)^{h,\perf}\preceq\ps{\FF_{p}}^{\perf},
    \end{align*}
    which is the perfect-hull version of the long sought-after 
\begin{question}
Is $\FF_{p}(t)^{h}$ an elementary substructure of $\ps{\FF_{p}}$?
\end{question}
\end{corollary}

% \subsection{Tilting adjectives}

\section{The Main Theorem of Jahnke and Kartas}

\begin{definition}
	We define the \define{sharp} maps
	$\sharp_{n}:K^{\flat}\rightarrow K$
	by $\sharp_{n}(x_{i})=x_{n}$, for each $n<\omega$.
	We write $\sharp=\sharp_{0}$
	and $x^{\sharp}=\sharp(x)$ for all $x\in K^{\flat}$.
	We define the \define{natural} map
	$\natural:K^{\flat}\rightarrow K^{\ult}$
	by $\natural(x_{i})=x^{\natural}:=(x_{i})^{\ult}$.
\end{definition}

So $\natural$ sends each $x=(x_{i})$ to the ultralimit of a compatible system of $p^{n}$-th roots of $x_{0}$;
it is literally the ultralimit of the maps $\sharp_{n}$.

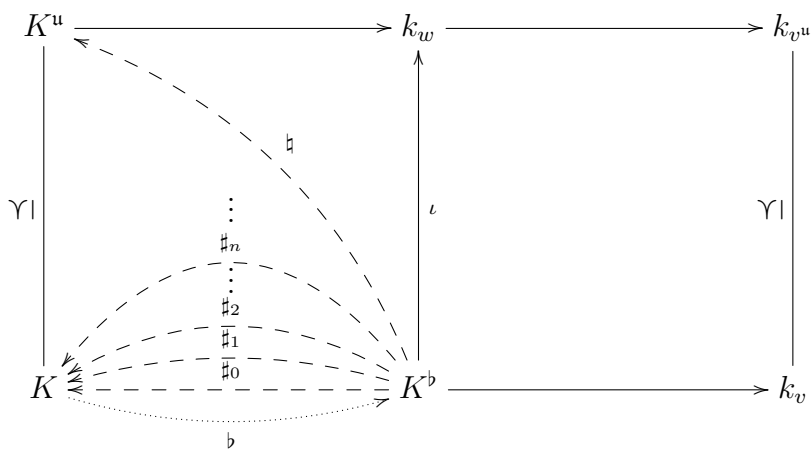
\begin{figure}[ht]
  \begin{align*}
  	\xymatrix@=10.0em{
		K^{\ult}
  	\ar@{->}[r]^{}
&
	  k_{w}
  	\ar@{->}[r]^{}
&
		k_{v^{\ult}}
\\
		K
	\ar@{-}[u]^{\rotatebox{90}{$\preceq$}}
	\ar@/_1.0pc/@{.>}[r]_{\flat}
	  % \ar@/_1.0pc/@[red][rr]
&
		K^{\flat}
	\ar@/_0.0pc/@{-->}[l]_{\sharp_{0}}
	\ar@/_1.0pc/@{-->}[l]_{\sharp_{1}}
	\ar@/_2.0pc/@{-->}[l]_{\sharp_{2}}
	  \ar@/_4.0pc/@{-->}[l]_{\sharp_{n}}^*--{\vdots}_*+++{\vdots}
	\ar@/_2.0pc/@{-->}[ul]_{\natural}
	\ar@{->}[u]_{\iota}
  	\ar@{->}[r]^{}
&
% 	  \mathcal{O}_{v^{\flat}}/\pi^{\flat}
%   	\ar@{->}[r]^{}
% &
		k_{v}
	\ar@{-}[u]^{\rotatebox{90}{$\preceq$}}
  	}
  \end{align*}
\caption{Illustration of the proof of Theorem~\ref{thm:first}}
\end{figure}

\begin{lemma}
	\
	\begin{enumerate}[{\bf(i)}]
		\item
			$\natural(K^{\flat\times})\subseteq\mathcal{O}_{w}^{\times}$
		\item
			$\natural(\mathcal{O}_{K^{\flat}})\subseteq\mathcal{O}_{v^{\ult}}$
		\item
			$\natural(\mathfrak{m}_{K^{\flat}})=\natural(\mathcal{O}_{K^{\flat}})\cap\mathfrak{m}_{v^{\ult}}$
		\item
			$\natural$ is multiplicative and additive modulo $p$
	\end{enumerate}
			Consequently
			$\iota:=\res_{w}\circ\natural$
			is a well-defined embedding of valued fields
			$(K^{\flat},v^{\flat})\rightarrow(k_{w},\bar{v}^{\ult})$.
\end{lemma}
This is relatively routine to verify.
It yields what I would describe as the first main theorem of \cite{JahnkeKartas}, which describes three surprising elementary embeddings of fields and valued fields.

In the following, a \define{distinguished uniformizer} in a valued field $(K,v)$ is an element of $\mathfrak{m}_{v}\setminus\{0\}\cap K^{(p^{\infty})}$.

\begin{theorem}[{First Main Theorem, \cite[Theorem 6.1.3]{JahnkeKartas}}]\label{thm:first}
	Let $(K,v)$ be a perfectoid field with a distinguished uniformizer $\pi$,
and
	let $\pi^{\flat}\in K^{\flat}$ be any preimage of $\pi$ under the map $\sharp_{0}$.
Let $\ult$ be a nonprincipal ultrafilter on $\NN$.
	Let $w$ be the coarsest coarsening of $v^{\ult}$ such that $w(\pi)>0$.
Then there are three elementary embeddings:
	\begin{enumerate}[{\bf(i)}]
		\item
The map $\natural$ induces an $\Lring$-elementary embedding 
			\begin{align*}
				\alpha:
				\mathcal{O}_{v^{\flat}}/\pi^{\flat}
				&\rightarrow
				\mathcal{O}_{v^{\ult}}/\pi^{\flat\natural}
				\\
				x+(\pi^{\flat})
				&\mapsto
				x^{\natural}+(\pi^{\flat\natural}).
			\end{align*}
		\item
The map $\iota$ induces an $\Lring$-elementary embedding 
			\begin{align*}
				\bar{\iota}:
				\mathcal{O}_{v^{\flat}}/\pi^{\flat}
				&\rightarrow
				\mathcal{O}_{\bar{v}^{\ult}}/\res_{w}\pi^{\flat\natural}
				\\
				x+(\pi^{\flat})
				&\mapsto
				\iota x+(\res_{w}\pi^{\flat\natural}).
			\end{align*}
		\item
			The map $\iota:(K^{\flat},v^{\flat})\rightarrow(k_{w},\bar{v}^{\ult})$ is itself an $\Lval$-elementary embedding.
	\end{enumerate}
\end{theorem}
\begin{proof}
	\
	\begin{enumerate}[{\bf(i)}]
		\item
			By hypothesis the Frobenius map $x\mapsto x^{p}$ induces a surjection
			$$\Phi:\mathcal{O}_{v}/\pi\rightarrow\mathcal{O}_{v}/\pi.$$
			For each $i$ write $\Phi^{i}=\Phi\circ\ldots\circ\Phi$ for the $i$-fold iterate of $\Phi$,
			which is also surjective and has kernel $(\pi^{\flat}_{(n)})\unlhd\mathcal{O}_{v}$.
			Then the ultralimit
			$$\Phi_{\infty}=\ulim_{i\rightarrow\ult}\Phi_{i}:\mathcal{O}_{v^{\ult}}/\pi\rightarrow\mathcal{O}_{v^{\ult}}/\pi$$
			is a surjection 
			with kernel
			$(\pi^{\flat\natural})\unlhd\mathcal{O}_{v^{\ult}}$.
			As usual, we get an induced isomorphism
			$$\bar{\Phi}_{\infty}:\mathcal{O}_{v^{\ult}}/\pi^{\flat\natural}\rightarrow\mathcal{O}_{v^{\ult}}/\pi.$$
			Finally
			\begin{align*}
			\xymatrix@=2.0em{
				\mathcal{O}_{v^{\flat}}/\pi^{\flat}
				\ar@{->}[r]^{=}
			&
				\mathcal{O}_{v}/\pi
			\ar@{->}[r]^{\preceq}_{\Delta}
			&
				\mathcal{O}_{v^{\ult}}/\pi
			\ar@{->}[r]^{\sim}_{\bar{\Phi}_{\infty}^{-1}}
			&
				\mathcal{O}_{v^{\ult}}/\pi^{\flat\natural}
			}
			\end{align*}
			Here the first map is an isomorphism from the definition of tilting, and so is arguably an equality,
			the second map is the diagonal embedding, which is an elementary embedding by Łoś's Theorem,
			and the third map is the inverse of $\bar{\Phi}_{\infty}$.
			Writing $\alpha=\bar{\Phi}_{\infty}^{-1}\circ\Delta$ we have the required elementary embedding.
		\item
			The desired map
			$\bar{\iota}$
			is also the composition
			\begin{align*}
				\mathcal{O}_{v^{\flat}}/\pi^{\flat}
				\overset{\alpha}{\longrightarrow}
				\mathcal{O}_{v^{\ult}}/\pi^{\flat\natural}
				\longrightarrow
				\mathcal{O}_{\bar{v}^{\ult}}/\res_{w}\pi^{\flat\natural}
			\end{align*}	
			where the second map is the canonical isomorphism given by quotienting by the ideal $\mathfrak{m}_{w}$.
			Observe also that $\bar{\iota}$ is induced by $\iota$.
		\item
			We apply Theorem~\ref{thm:AKE-elementary_substructure}:
            since,
			both $\Gamma_{v^{\flat}}$ and $\Gamma_{\bar{v}^{\ult}}$ are regularly dense
			and since
			$\bar{\iota}:\mathcal{O}_{v^{\flat}}/\pi^{\flat}\overset{\preceq}{\longrightarrow}\mathcal{O}_{\bar{v}^{\ult}}/\res_{w}\pi^{\flat\natural}$
			provides the elementary embedding,
            we can conclude.
			\qedhere
	\end{enumerate}
\end{proof}

\begin{theorem}[{Second Main Theorem, \cite[Theorem 6.2.3]{JahnkeKartas}}]\label{thm:second}
	Let $(K,v,\pi)$ be a pointed perfectoid field
	and
	let $\ult$ be a nonprincipal ultrafilter on $\NN$
	and let $(K^{\ult},v^{\ult})$ be the corresponding ultrapower.
	Let $w$ be the coarsest coarsening of $v^{\ult}$ such that $w(\pi)>0$.
	Then:
	\begin{enumerate}[{\bf(i)}]
 		\item
			Every finite extension of $(K^{\ult},w)$ is unramified.
		\item
			The tilt $(K^{\flat},v^{\flat})$ elementarily embeds into $(k_{w},\bar{v})$.
		\item
			The equivalence of categories
			$K^{\ult}\fet\cong k_{w}\fet$
			restricts to
			$K\fet\cong K^{\flat}\fet$,
            where $F\fet$ denotes the category of finite étale extensions of $F$.
 	\end{enumerate}
\end{theorem}
\begin{proof}
	The first point, {\bf(i)}, is the ``Non-standard Tate/Gabber--Ramero'' Theorem:
    it follows from the fact that the value group $wK^{\ult}$ is divisible, and that $w$ itself is algebraically maximal by Proposition~\ref{prp:elementary}.
    Point
	{\bf(ii)} comes from point {\bf(iii)} of Theorem~\ref{thm:first},
    which proves that $\iota$ is elementary.
    % Finally, to prove {\bf(iii)} we observe that $(K^{\ult},w)$
\end{proof}

% \subsection{Quick corollaries}
% --- tameable
% $\Cperf$
% $\Tperf$
% --- axiomatizing algebraic maximality upstairs
% $(K,v,t)$ - alg max of $v_{t+}$
% --- tilting
% --- untilting and the Fargues--Fontaine curve $C_{\mathcal{F}}$
% --- AKE for perfectoids
% --- Almost purity and Fontaine--Wintenberger
% \cite{FontaineWintenberger}

\section{The role of deeply ramified fields}

\begin{definition}
	A valued field $(K,v)$ is
	\define{deeply ramified}
	if the module of Kähler differentials 
	$\Omega_{K^{\sep}/K}$
    is trivial.
\end{definition}

A valued field $(K,v)$ of 
residue characteristic of $(K,v)$ is $p>0$,
is deeply ramified if and only if
it is dense in its perfect hull.
In turn, this holds 
if and only if
the homomorphism
\begin{align*}
	\mathcal{O}_{\hat{v}}/p&\rightarrow\mathcal{O}_{\hat{v}}/p\\
	x&\mapsto x^{p}
\end{align*}
is surjective, where $\mathcal{O}_{\hat{v}}$ is the valuation ring of the completion $(\hat{K},\hat{v})$ of $(K,v)$,
see \cite[Theorem1.2]{KuhlmannRzepka}.
Kuhlmann and Rzepka also give the narrower definition of a {\em semitame} valued field.
A semitame valued field is deeply ramified, it has residue characteristic exponent $p$, and its value group is $p$-divisible.

\begin{theorem}[{cf~\cite[Theorem 1.2]{KuhlmannRzepka}}]
			Let $(K,v)$ be a nontrivially valued field of residue characteristic $p>0$.
	\begin{enumerate}[{\bf(i)}]
		\item
The following logical relations hold between the properties of $(K,v)$:
	\begin{align*}
		\text{\rm tame}
		\implies
		\text{\rm separably tame}
		\implies
		\text{\rm semitame}
		\implies
		\text{\rm deeply ramified}.
		% \implies
		% \text{\rm rdr field}.
	\end{align*}
\item
If $(K,v)$ is of rank $1$,
then it is semitame if and only if it is deeply ramified.
\item
	If the characteristic of $K$ is $p$,
    then the following are equivalent:
	\begin{enumerate}[{\bf(a)}]
		\item
		$(K,v)$
		is semitame,
		\item
		$(K,v)$
		is deeply ramified,
		\item
		the completion of $(K,v)$ is perfect,
		\item
		$(K,v)$
		is dense in its perfect hull,
		\item
		$K^{(p)}$ is dense in $(K,v)$.
	\end{enumerate}
\item
If $K$ is perfect and of positive characteristic then it is semitame.
	\end{enumerate}
\end{theorem}

\section{Building on these foundations}

I am aware of several fronts on which the model theory is advancing
by broadening or varying the ``almost tame'' framework.

% \subsection{Ketelsen's tilting of theories}

Ketelsen has announced
(\cite{Ketelsentalk})
an extension of the tilting functor (staying in the world of fields of mixed characteristic) from perfectoid fields to any model of the theory $\TST$ of semitame fields of mixed characteristic.
At the time of writing I understand this will appear in Ketelsen's forthcoming PhD thesis \parencite{Ketelsenthesis}.

\begin{definition}
	The class of semitame valued fields of mixed characteristic $(0,p)$,
    denoted by
    $\CST_{(0,p)}$
    is the class of valued fields
    $(K,v)$ that satisfy
	\begin{enumerate}
		\item
		the valued field $(K,v)$ is henselian of mixed characteristic $(0,p)$,
		\item
		the value group $vK$ is $p$-divisible, and
		\item
		the ring $\Ring{v}/p$ is semiperfect.
	\end{enumerate}
\end{definition}

For any ``sufficiently%
\footnote{Here $\aleph_{1}$-saturated suffices.}
saturated''
model
$(K,v)\in\CST_{(0,p)}$
we may define
$(K^{\flat},v^{\flat})$
to be
\begin{align*}
	\Hs{(Kv_{p+})^{\flat}}{\Gamma},
\end{align*}
with valuation $\tilt{v}$ given by the composition
$\tilt{(v_{p-})}\circ v_{t}$.
Ketelsen showed that given any $(K,v),(L,w)\models\TST_{(0,p)}$
we have the implication
\begin{align*}
(K,v)\equiv(L,w)&\implies(\tilt{K},\tilt{v})\equiv(\tilt{L},\tilt{w}),
\end{align*}
which is a form of \AKEtext.
Rephrased in the framework of \cite{AF_fragments},
that tilting is an interpretation
{\em for sentences}
of the theory of almost tame valued fields of positive characteristic in the theory of semitame valued fields of mixed characteristic.

\begin{theorem}[{cf \cite{Ketelsentalk}}]
This yields a well-defined map
$T\mapsto\tilt{T}$
from theories extending $\TAT_{(0,p)}$
to theories extending $\TAT_{(p,p)}$,
such that $(K,v)\models T\implies(K,v)^{\flat}\models T^{\flat}$.
\end{theorem}

The previous theorem has the same computable sentence-to-sentence variant that one expects.
Consequently:

\begin{corollary}
	Let $T$ be any theory (not necessarily complete) extending $\TAT_{(0,p)}$.
	Then
	$\tilt{T}\mred T$.
    In particular, if $T$ is decidable then $\tilt{T}$ is also decidable.
\end{corollary}

% \subsection{Separable taming}

Jahnke and van der Schaaf have also extended the Jahnke--Kartas argument to account for ``separably taming'';
at the time of writing I understand this to be work in progress,
but it appears also in the Master's thesis \parencite{vanderSchaaf}.
More precisely, this replaces the theory of tame valued fields with the theory of separably tame valued fields.

\begin{definition}[{``Almost separably tame fields'', \cite{vanderSchaaf}}]
The class
$\CAST_{p}$
is the class of pointed valued fields $(K,v,\pi)$
of equal characteristic $p$
such that:
	\begin{enumerate}[{\bf(i)}]
		\item
			$(K,v)$ is henselian and deeply ramified
		\item
			The valuation ring $\Ring{v}[\pi^{-1}]$ is separably algebraically maximal.
\end{enumerate}
\end{definition}

Elements of $\CAST_{p}$ are to be viewed as $\Lvalpi$-structures.

\begin{theorem}[{\cite{vanderSchaaf}}]
The class $\CAST$ is an $\Lvalpi$-elementary class.
\end{theorem}
\begin{proof}[Idea of the proof]
If $(K,v,\pi)$ is henselian and deeply ramified of equal characteristic $p$ then 
$\Ring{v}[\pi^{-1}]$ is separably defectless
if and only if
\begin{itemize}
\item
for every finite separable unramified extension $(L,w)/(K,v)$
there exists $a\in\Ring{v}$ such that $L=K(a)$
and $0\leq v(\delta(a))\leq v(\pi)$.
\end{itemize}
This property is clearly expressible as the conjunction of infinitely many $\Lvalpi$-sentences.
\end{proof}

\begin{theorem}[{\cite{vanderSchaaf}}]
The class $\CAST$ is an $\AKE$-class in the same senses as
Theorems~\ref{thm:AKE-relative_subcompleteness},~\ref{thm:AKE-elementary_substructure},~\ref{thm:AKE-elementary_equivalence}, and~\ref{thm:AKE-existential_closedness}.
\end{theorem}
The proof follows the same path as the theorems for $\CAT$, but replacing $\TVF$ by $\STVF$ and applying Theorem~\ref{thm:STVFe}.

\begin{corollary}
For any perfect field $k$, we have
$k(t)^{h,\mathrm{AS}}\preceq\ps{k}^{\mathrm{AS}}$,
where the latter is the Artin--Schreier closure of $\ps{k}$ and the former is the relative algebraic closure of $k(t)$ in the latter.
\end{corollary}

% \subsection{Kartas' $C_{i}$ transfer}
% A general perspective,
% in fragments and uniformly
% \cite{Kartas_perfectoid-Ci}
% \subsection{Perfectoid fields in continuous logic}
% Rideau-Kikuchi--Scanlon--Simon
% Questions for the future
% \subsection{Almost mathematics}
% \cite{FontaineWintenberger}
% \cite{GabberRamero}

%%%%%%%%%%%%%%%%%%%%%%%%%%%%%%%%%%%%%%%%%%%%%%%%%
% \clearpage
%%%%%%%%%%%%%%%%%%%%%%%%%%%%%%%%%%%%%%%%%%%%%%%%%

% \vfill

\section*{Acknowledgements}

The author would like to thank Franziska Jahnke, Konstantinos Kartas, Margarete Ketelsen, and Silvain Rideau-Kikuchi for comments on an earlier version.
Sincere thanks are due to an anonymous referee whose careful feedback enriched the manuscript greatly.
This work may look like a survey, but it is far from comprehensive.
Nevertheless many errors, omissions, and simplifications remain
--- for all of these the author has sole responsibility.

During the preparation of this manuscript the author 
was supported by
the ANR-DFG project ``AKE-PACT'' (ANR-24-CE92-0082)
and
by ``Investissement d'Avenir'' launched by the French Government and implemented by ANR (ANR-18-IdEx-0001) as part of its program ``Emergence''.

%%%%%%%%%%%%%%%%%%%%%%%%%%%%%%%%%%%%%%%%%%%%%%
%%%%%%%%%%%%%%%%%%%%%%%%%%%%%%%%%%%%%%%%%%%%%%
\printbibliography
%%%%%%%%%%%%%%%%%%%%%%%%%%%%%%%%%%%%%%%%%%%%%%
%%%%%%%%%%%%%%%%%%%%%%%%%%%%%%%%%%%%%%%%%%%%%%
\end{document}